\documentclass[10pt]{article}
        
\usepackage{amssymb,amsmath,amsfonts,amsthm}
\usepackage{graphicx}
\usepackage[a4paper,includeheadfoot,left=4cm,right=4cm,top=3cm,bottom=3cm]{geometry}

\usepackage[font={small}]{caption}

\usepackage[numbers,comma,sort&compress]{natbib}
\makeatletter
\let\@fnsymbol\@arabic
\makeatother

\usepackage{color}
\definecolor{marin}{rgb}{0.,0.3,0.7}
\usepackage[colorlinks,citecolor=marin,linkcolor=marin,urlcolor=marin,bookmarksopen,bookmarksnumbered]{hyperref}

\newcommand{\si}{\sigma}
\newcommand{\ph}{\varphi}
\newcommand{\Om}{\Omega}
\newcommand{\Ic}{\mathcal{I}}
\providecommand{\abs}[1]{\lvert#1\rvert}
\providecommand{\absbig}[1]{\bigl\lvert#1\bigr\rvert}

\providecommand{\kla}[1]{(#1)}
\providecommand{\klabig}[1]{\bigl(#1\bigr)}
\providecommand{\klaBig}[1]{\Bigl(#1\Bigr)}
\providecommand{\klabigg}[1]{\biggl(#1\biggr)}
\providecommand{\klaBigg}[1]{\Biggl(#1\Biggr)}

\providecommand{\norm}[1]{\lVert#1\rVert}
\providecommand{\normbig}[1]{\bigl\lVert#1\bigr\rVert}
\providecommand{\normBig}[1]{\Bigl\lVert#1\Bigr\rVert}
\providecommand{\normbigg}[1]{\biggl\lVert#1\biggr\rVert}

\providecommand{\normv}[1]{\ensuremath{{\lVert\hskip-1pt\lvert}#1{\rvert\hskip-1pt\rVert}}}
\providecommand{\normvbig}[1]{\ensuremath{{\bigl\lVert\hskip-1pt\bigl\lvert}#1{\bigr\rvert\hskip-1pt\bigr\rVert}}}

\newcommand\myfor{\qquad\text{for}\qquad}
\newcommand\myand{\qquad\text{and}\qquad}
\newcommand{\sfrac}[2]{\mbox{$\textstyle\frac{#1}{#2}$}}
\newcommand{\iu}{\mathrm{i}}
\newcommand{\e}{\mathrm{e}}
\newcommand{\drm}{\mathrm{d}}
\DeclareMathOperator{\ReT}{Re}

\DeclareMathOperator{\sinc}{sinc}
\newtheorem{theorem}{Theorem}[section]
\newtheorem{lemma}[theorem]{Lemma}
\newtheorem{proposition}[theorem]{Proposition}

\numberwithin{equation}{section}

\title{On a splitting method for the Zakharov system}

\author{Ludwig Gauckler\,\thanks{Institut f\"ur Mathematik,
          Freie Universit\"at Berlin,
          Arnimallee 9,
          D-14195 Berlin, Germany
          ({\tt gauckler@math.fu-berlin.de}).}
}

\date{Version of 15 November 2017}

\begin{document}

\maketitle

\begin{abstract}
An error analysis of a splitting method applied to the Zakharov system is given. The numerical method is a Lie--Trotter splitting in time that is combined with a Fourier collocation in space to a fully discrete method. First-order convergence in time and high-order convergence in space depending on the regularity of the exact solution are shown for this method. The main challenge in the analysis is to exclude a loss of spatial regularity in the numerical solution. This is done by transforming the numerical method to new variables and by imposing a natural CFL-type restriction on the discretization parameters.\\[1.5ex]
\textbf{Mathematics Subject Classification (2010):} 
65M15, 
65P10, 
65M20.\\[1.5ex] 
\textbf{Keywords:} Zakharov system, splitting method, Lie--Trotter splitting, Fourier collocation, error bounds, loss of derivatives.
\end{abstract}

\section{Introduction}

We consider the \emph{Zakharov system}
\begin{equation}\label{eq-zak-orig}\begin{split}
\iu \partial_t \psi &= -\Delta \psi + u\psi,\\
\partial_{tt} u &= \Delta u + \Delta \abs{\psi}^2
\end{split}\end{equation}
introduced by Zakharov \cite{Zakharov1972}, which describes the propagation of Langmuir waves in a plasma. It is a Schr\"odinger equation for the complex-valued function $\psi=\psi(x,t)$ that is nonlinearly coupled to a wave equation for the real-valued function $u=u(x,t)$. This system is often considered on the full space $\mathbb{R}^d$, but for a subsequent numerical discretization it is usually truncated to a finite box. We therefore consider this system with periodic boundary conditions in $d$-dimen\-sio\-nal space, with period normalized to $2\pi$, i.e., $x\in\mathbb{T}^d=\mathbb{R}^d/(2\pi\mathbb{Z}^d)$. 

As a numerical method for the Zakharov system \eqref{eq-zak-orig}, we study the method of Jin, Markowich \& Zheng \cite{Jin2004}, which is a Lie--Trotter splitting in time combined with a Fourier collocation in space. This method has proven to work well in extensive numerical tests, see \cite{Jin2004}. 
In the present paper, we give a rigorous error analysis of this method, showing first-order convergence in time and high-order convergence in space depending on the regularity of the exact solution. The error bound holds under a CFL-type step-size restriction on the discretization parameters. In view of the proven temporal and spatial error bounds, this CFL condition is a natural restriction on the discretization parameters. As will be illustrated by numerical experiments, this condition is not only sufficient but also necessary for convergence. 

\medskip

At first glance, an error analysis of splitting methods applied to \eqref{eq-zak-orig} might seem an easy exercise in view of the available error analysis for splitting methods applied to semilinear Schr\"odinger equations \cite{Lubich2008a,Thalhammer2012a} and semilinear wave equations \cite{Gauckler2015}. This, however, is not the case, essentially because there is a formal \emph{loss of spatial regularity} (loss of spatial derivatives) in the Zakharov system that a numerical method and its analysis have to handle. In fact, an inspection of the wave equation in \eqref{eq-zak-orig} suggests that $(u(\cdot,t),\partial_t u(\cdot,t))$ is in the Sobolev space $H^{s+1}(\mathbb{T}^d)\times H^s(\mathbb{T}^d)$ for some $s>\frac{d}2$ only if $\psi(\cdot,t)$ is in the Sobolev space $H^{s+2}(\mathbb{T}^d)$ of higher order. An inspection of the Schr\"odinger equation in \eqref{eq-zak-orig} suggests, however, that $\psi(\cdot,t)$ has no higher regularity than $u(\cdot,t)$, which is $H^{s+1}(\mathbb{T}^d)$. See also the introduction of \cite{Herr}. We emphasize that this loss of spatial regularity is formal in the sense that it only appears in a naive and formal analysis, while it can be shown, with more sophisticated arguments, to not affect the actual (exact) solution.

The formal loss of spatial regularity is a major difficulty in the numerical analysis of the Zakharov system. The problem is that a loss of regularity might be also present in a numerical method, or at least in its (naive) analysis. Typically, \emph{implicit} or \emph{semi-implicit} methods can be designed to avoid this potential loss of regularity, and indeed they have been introduced and analyzed for the Zakharov system in \cite{Glassey1992a,Glassey1992,Chang1994,Chang1995}. More recently, however, several \emph{explicit} splitting methods have been introduced \cite{Bao2005a,Bao2003a,Jin2004,Jin2006}, among them the method of Jin, Markowich \& Zheng \cite{Jin2004} that we consider here. These explicit methods performed well in numerous numerical test, but rigorous error bounds for them seem to be missing so far. In fact, direct estimates for such methods, as used in the mentioned papers on Schr\"odinger and wave equations, cannot exclude a loss of spatial regularity of the numerical solution in each time step, see Section~\ref{sec-exp}. Moreover, a loss of spatial regularity actually occurs when the mentioned CFL condition is violated, see also Section~\ref{sec-exp}. 

\medskip

In the present paper, we give an error analysis for such an explicit method. The main tool is a transformation of the numerical solution to new variables. In these new variables, a loss of spatial regularity can be excluded and an error analysis is in fact possible, leading finally to error bounds in the original variables. The transformation that we use is a discrete analogon of a transformation introduced by Ozawa~\& Tsutsumi in \cite{Ozawa1992} for the analysis of the Zakharov system itself. We mention that this latter transformation has been put to numerical use by Herr \& Schratz in \cite{Herr} in a conceptually different way than here. They start from the Zakharov system in the new variables of \cite{Ozawa1992} and design and analyze a new numerical method for the system in these new variables. In contrast, we translate here the transformation of \cite{Ozawa1992} to a discrete level and show how this helps to analyze an existing and well-established numerical method. 

The techniques developed in the present paper can also be used to analyze the extension of the considered Lie--Trotter splitting to a (formally) second-order Strang splitting for the Zakharov system, as considered also in \cite{Jin2004}. We expect that the techniques can in addition be used to prove error bounds for the splitting integrators of Bao, Sun \& Wei \cite{Bao2005a,Bao2003a} for the Zakharov system. Moreover, we expect them to be useful to analyze numerical methods for other equations with a formal loss of spatial derivatives, for which a transformation to new variables without loss of derivatives is available for the exact solution, for example for derivative nonlinear Schr\"odinger equations~\cite{Hayashi1993}.
From a technical point of view, the present paper provides a new technique to analyze (explicit) numerical methods for hyperbolic equations, where a formal loss of spatial regularity is problematic. Previous techniques for that purpose include energy estimates used for quasilinear wave equations \cite{Gaucklera,Hochbruck} and special regularity results for the inviscid Burgers equation used for the KdV equation \cite{Holden2011} and for partial differential equations with Burgers nonlinearity \cite{Holden2013}. 

\medskip

The paper is organized as follows. In Section \ref{sec-methods}, the considered numerical method is introduced and the error bounds are stated. In Section \ref{sec-transformation}, a discrete version of the transformation of \cite{Ozawa1992} is described for the Lie--Trotter splitting, which is then used to prove the corresponding error bounds for the spatial semi-discretization in Section~\ref{sec-proof-semi} and for the fully discrete method in Section \ref{sec-proof}. 
In the final Section \ref{sec-exp}, numerical experiments are presented that illustrate amongst others the loss of spatial regularity when the CFL condition is not met. 

\medskip

\noindent\textbf{Notation.} For $s\ge 0$, we let $H^s=H^s(\mathbb{T}^d)$ denote the standard Sobolev space equipped with the norm
\[
\norm{v}_s = \klabigg{ \sum_{j\in\mathbb{Z}^d} \max\klabig{\abs{j},1}^{2s} \abs{v_j}^2 }^{1/2} \myfor v(x) = \sum_{j\in\mathbb{Z}^d} v_j\, \e^{\iu j\cdot x},
\]
where $j\cdot x = j_1x_1+\dots+j_d x_d$ and $\abs{j}^2 = j\cdot j = j_1^2+\dots+j_d^2$.
We will make frequent use of the fact that this space forms, for $s>\frac{d}2$, a normed algebra:
\begin{equation}\label{eq-algebra}
\norm{vw}_s \le C \norm{v}_s \norm{w}_s
\end{equation}
with $C$ depending only on $d$ and $s$.
On the product $H^{s+1}\times H^s$ we use the norm
\[
\normv{(v,\dot{v})}_s = \klabig{ \norm{v}_{s+1}^2 + \norm{\dot{v}}_s^2 }^{1/2}, \qquad (v,\dot v) \in H^{s+1}\times H^s.
\]

We further denote by $\Om$ the operator on functions that multiplies the $j$th Fourier coefficient of its argument by $\abs{j}=\sqrt{j_1^2+\dots+j_d^2}$:
\begin{equation}\label{eq-Om}
(\Om v) (x) = \sum_{j\in\mathbb{Z}^d} \abs{j} v_j\, \e^{\iu j\cdot x} \myfor v(x) = \sum_{j\in\mathbb{Z}^d} v_j\, \e^{\iu j\cdot x}.
\end{equation}
In other words, we have $\Om^2=-\Delta$, and we will often favour the notation $\Om^2$ over $-\Delta$ since the numerical method eventually involves functions of $\Om$. 

For functions depending on space $x$ and time $t$, we write $u(t)=u(\cdot,t)$ in the following. 
In addition, we use the function
\begin{equation}\label{eq-phifunc}
\phi(\xi) = \frac{\e^{\xi}-1}{\xi},
\end{equation}
which is familiar from exponential integrators.
We will often use that $\e^\xi-\e^\eta= (\xi-\eta) \phi(\xi-\eta) \e^\eta$.

\section{Numerical method and error bounds}\label{sec-methods}

Throughout the paper, we consider the Zakharov system \eqref{eq-zak-orig} in first-order form 
\begin{equation}\label{eq-zak}\begin{split}
\iu \partial_t \psi &= -\Delta \psi + u\psi,\\
\partial_{t} u &= \dot{u},\\
\partial_{t} \dot{u} &= \Delta u + \Delta \abs{\psi}^2.
\end{split}\end{equation}
For the numerical discretization of this system, we consider a numerical method introduced in \cite{Jin2004}, which combines a splitting integrator in time with a Fourier collocation in space.

\subsection{Fourier collocation in space}

We first discretize \eqref{eq-zak} in space by Fourier collocation. This method is based on the ansatz space
\[
\Biggl\{ \, \sum_{j\in\mathcal{K}} v_j \, \e^{\iu j\cdot x} : v_j\in\mathbb{C} \,\Biggr\},\qquad \mathcal{K} = \bigl\{ -K,\dots,K-1\bigr\}^d,
\]
of trigonometric polynomials of degree $K$. We replace $\psi(t)$, $u(t)$ and $\dot{u}(t)$ in \eqref{eq-zak} by trigonometric polynomials $\psi^K(t)$, $u^K(t)$ and $\dot{u}^K(t)$ from this ansatz space. As such trigonometric polynomials are uniquely determined by their values in the discrete points $x_k=k\pi/K$, $k\in\mathcal{K}$, we require that these ansatz functions satisfy the Zakharov system \eqref{eq-zak} in these discrete points. 
Letting $\Ic$ denote the trigonometric interpolation of degree $K$, i.e., the operator that assigns to a function the unique trigonometric polynomial of degree $K$ that takes the same values in the discrete points $x_k$, $k\in\mathcal{K}$, the system of equations for $\psi^K(t)$, $u^K(t)$ and $\dot{u}^K(t)$ then reads
\begin{equation}\label{eq-zakK}\begin{split}
\iu \partial_t \psi^K &= -\Delta \psi^K + \Ic \klabig{u^K\psi^K},\\
\partial_{t} u^K &= \dot{u}^K,\\
\partial_{t} \dot{u}^K &= \Delta u^K + \Delta \Ic\klabig{\abs{\psi^K}^2}
\end{split}\end{equation}
with initial values 
\[
\psi^K(t_0) = \Ic\klabig{\psi(t_0)}, \qquad u^K(t_0) = \Ic\klabig{u(t_0)}, \qquad \dot{u}^K(t_0) = \Ic\klabig{\dot{u}(t_0)} .
\]
In this system, $u^K$ and $\dot{u}^K$ are not necessarily real-valued, but they take real values in the discrete points $x_k$, $k\in\mathcal{K}$. 

We recall here the following well-known and fundamental lemma on the trigonometric interpolation $\Ic$.

\begin{lemma}\label{lemma-inter}
Let $s>\frac{d}2$ and $\si\ge 0$. Then we have
\[
\norm{\Ic(v)}_{s} \le C \norm{v}_{s}, \qquad v\in H^s,
\]
and
\[
\norm{v-\Ic(v)}_{s} \le C K^{-\si} \norm{v}_{s+\si}, \qquad v\in H^{s+\si}. 
\]
The constant $C$ depends only on $d$ and $s$ (and $\si$). \hfill \qed
\end{lemma}

\subsection{Lie--Trotter splitting in time}

The discretization in time of the semi-discretization in space \eqref{eq-zakK} is based on the splitting of \eqref{eq-zakK} into
\begin{align*}
\iu \partial_t \psi^K &= -\Delta \psi^K, && & \iu \partial_t \psi^K &= \Ic\klabig{u^K\psi^K},\\
\partial_{t} u^K &= 0, && \text{and} & \partial_{t} u^K &= \dot{u}^K,\\
\partial_{t} \dot{u}^K &= 0 && & \partial_{t} \dot{u}^K &= \Delta u^K + \Delta \Ic\klabig{ \abs{\psi^K}^2 }.
\end{align*}
As noted in \cite{Jin2004}, both of these splitted systems can be solved exactly: 
\begin{itemize}
\item The solution of the first system is $\psi^K(t)=\e^{\iu (t-t_0)\Delta} \psi^K(t_0) = \e^{-\iu (t-t_0)\Om^2} \psi^K(t_0)$, $u^K(t)=u^K(t_0)$ and $\dot{u}^K(t)=\dot{u}^K(t_0)$. 
\item For the second system, we first note the solution of its first equation is given by 
\[
\psi^K(t) = \Ic \klaBig{ \e^{-\iu (t-t_0) v^K(t-t_0)} \psi^K(t_0) }
\]
with
\begin{equation}\label{eq-v-int}
v^K(t)= \frac{1}{t} \int_{t_0}^{t_0+t} u^K(t')\,\drm t'. 
\end{equation}
Using that the values $u^K(x_k,t)$, $k\in\mathcal{K}$, in the wave equation are real, this shows in particular that the absolute values $\abs{\psi^K(x_k,t)}$, $k\in\mathcal{K}$, are constant in time. Hence, $\Ic(\abs{\psi^K}^2)$ is constant in time, and the wave equation in the second system can be reformulated as $\partial_{tt} (u^{K}+\Ic(\abs{\psi^{K}}^2)) = \Delta (u^{K}+\Ic(\abs{\psi^{K}}^2))$. 
With the matrix 
\begin{equation}\label{eq-R}
R(t) = \begin{pmatrix} \cos(t\Om) & t \sinc\kla{t\Om} \\ -\Om\sin\kla{t\Om} & \cos\kla{t\Om} \end{pmatrix},
\end{equation}
where $\Om$ is the operator \eqref{eq-Om},
the solution to the wave equation in the second system is thus given by 
\[
\begin{pmatrix} u^K(t) + \Ic\klabig{\abs{\psi^K(t_0)}^2}\\ \dot{u}^K(t) \end{pmatrix} = R\kla{t-t_0} \begin{pmatrix} u^K(t_0) + \Ic\klabig{\abs{\psi^K(t_0)}^2}\\ \dot{u}^K(t_0) \end{pmatrix}.
\]
\end{itemize}
The Lie--Trotter splitting applied to \eqref{eq-zakK} uses a composition of the described flows of the two splitted equations to compute approximations to \eqref{eq-zakK}. Denoting the time step-size $\tau$, the method computes trigonometric polynomials $\psi_{n+1}^K$, $u_{n+1}^K$ and $\dot{u}_{n+1}^K$ of degree $K$ that are supposed to approximate the solutions $\psi^K(t_{n+1})$, $u^K(t_{n+1})$ and $\dot{u}^K(t_{n+1})$ of \eqref{eq-zakK} at discrete times $t_{n+1}= t_0 + (n+1)\tau$:
\begin{subequations}\label{eq-split}
\begin{equation}\label{eq-split-1}\begin{split}
\psi_{n+1}^K &= \e^{-\iu\tau\Om^2} \Ic \klabig{\e^{-\iu \tau v^K_{n+1}} \psi_{n}^K},\\
\begin{pmatrix} u_{n+1}^K\\ \dot{u}_{n+1}^K \end{pmatrix} &= R\kla{\tau} \begin{pmatrix} u_{n}^K\\ \dot{u}_{n}^K \end{pmatrix}
+ \klabig{R\kla{\tau}-1}  \begin{pmatrix} \Ic \klabig{\abs{\psi_{n}^K}^2} \\ 0 \end{pmatrix} ,
\end{split}\end{equation}
where (by computing the integral in \eqref{eq-v-int})
\begin{equation}\label{eq-vn-1}
v_{n+1}^K = \sinc\kla{\tau\Om} u_n^K + \sfrac12 \tau \sinc\klabig{\sfrac12 \tau\Om}^2 \dot{u}_n^K + \klabig{\sinc\kla{\tau\Om}-1} \Ic \klabig{\abs{\psi_{n}^K}^2}.
\end{equation}
Initial values are computed by trigonometric interpolation,
\begin{equation}\label{eq-split-start}
\psi_0^K = \psi^K(t_0) = \Ic\klabig{ \psi(t_0) }, \quad u_0^K = u^K(t_0)=\Ic \klabig{ u(t_0) }, \quad \dot{u}_0^K = \dot{u}^K(t_0) = \Ic \klabig{ \dot{u}(t_0) }.
\end{equation}
\end{subequations}
We note that
\begin{equation}\label{eq-vn-2}
v_{n+1}^K = \sinc(\tau\Om) u_{n+1}^K - \sfrac12 \tau \sinc\klabig{\sfrac12 \tau\Om}^2 \dot{u}_{n+1}^K + \klabig{ \sinc(\tau\Om)-1 } \Ic \klabig{\abs{\psi_{n}^K}^2},
\end{equation}
which can be verified by inserting \eqref{eq-split-1} into \eqref{eq-vn-2} and using a lot of trigonometric identities.

\subsection{Statement of error bounds}\label{subsec-results}

We state our error bounds for the spatial semi-discretization \eqref{eq-zakK} by Fourier collocation and for the full discretization \eqref{eq-split} by Lie--Trotter splitting. For these global error bounds on finite time intervals, we assume regularity of the exact solution to the Zakharov system \eqref{eq-zak}:
\begin{equation}\label{eq-regularity}
\norm{\psi(t)}_{s+2+\si} + \normv{(u(t),\dot{u}(t))}_{s+\si} \le M \myfor 0\le t-t_0\le T 
\end{equation}
for some $s>\frac{d}2$ and $\si>0$. This regularity assumption can be expected to hold locally in time by the well-posedness theory of the Zakharov system \eqref{eq-zak} on the torus, see \cite{Kishimoto2013} and references therein.

For the spatial semi-discretization \eqref{eq-zakK}, we then have the following error bound, whose proof is given in Sections \ref{sec-transformation} and \ref{sec-proof-semi} below.

\begin{theorem}\label{thm-semi}
Let $s>\frac{d}2$ and $\si> 0$, and assume that the exact solution to \eqref{eq-zak} satisfies~\eqref{eq-regularity} 
with these $s$ and $\si$ and with $M\ge 1$ and $T>0$. Then, the error of the Fourier collocation \eqref{eq-zakK} with spatial discretization parameter $K\ge K_0$ is bounded by
\[
\normbig{\psi(t) - \psi^K(t)}_{s+2} + \normvbig{\klabig{u(t),\dot{u}(t)} - \klabig{u^K(t),\dot{u}^K(t)}}_{s} \le C K^{-\si}
\]
for $0\le t-t_0 \le T$.
The constants $C$ and $K_0$ depend only on $s$, $\si$, $M$ and $T$ of \eqref{eq-regularity} and on the dimension $d$.
\end{theorem}

Our main result is the following global error bound for the Lie--Trotter splitting~\eqref{eq-split}. Its proof is given in Sections \ref{sec-transformation} and \ref{sec-proof} below. The result holds under the CFL-type step-size restriction
\begin{equation}\label{eq-cfl}
d \tau K^2 \le c < 2 \pi
\end{equation}
on the time step-size $\tau$, the spatial discretization parameter $K$ and the dimension $d$.

\begin{theorem}\label{thm-main}
Let $s>\frac{d}2$ and $\si\ge 2$, and assume that the exact solution to \eqref{eq-zak} satisfies~\eqref{eq-regularity} 
with these $s$ and $\si$ and with $M\ge 1$ and $T>0$. Then, the global error of the Lie--Trotter splitting \eqref{eq-split} with time step-size $\tau\le\tau_0$ and spatial discretization parameter $K\ge K_0$ that satisfy the CFL-type step-size restriction \eqref{eq-cfl} is bounded by
\[
\normbig{\psi(t_n) - \psi_n^K}_{s+2} + \normvbig{\klabig{u(t_n),\dot{u}(t_n)} - (u_n^K,\dot{u}_n^K)}_s \le C \klabig{\tau + K^{-\si}}
\]
for $0\le t_n-t_0=n\tau \le T$.
The constants $C$, $K_0$ and $\tau_0$ depend on $s$, $\si$, $M$ and $T$ of \eqref{eq-regularity}, on the dimension $d$ and on $c$ of \eqref{eq-cfl}.
\end{theorem}

In view of the error bound of Theorem \ref{thm-main}, it is natural to choose $\tau=\mathcal{O}(K^{-\si})=\mathcal{O}(K^{-2})$. The CFL condition \eqref{eq-cfl} of Theorem \ref{thm-main} is thus not severe but natural.

\section{New variables for the numerical method}\label{sec-transformation}

In this section, we present the transformation from \cite{Ozawa1992} of the Zakharov system \eqref{eq-zak} to new variables, in which a loss of spatial derivatives can be excluded. We then show, how the semi-discretization in space \eqref{eq-zakK} and the fully discrete splitting method~\eqref{eq-split} can be transformed in a similar way to new variables. As it turns out, a loss of spatial derivatives can be excluded in these new variables as well. The error bounds of Theorems \ref{thm-semi} and \ref{thm-main} are therefore proven in Sections \ref{sec-proof-semi} and \ref{sec-proof} below in these new variables.

\subsection{Transformation of the exact solution}

We describe the transformation of the Zakharov system \eqref{eq-zak} introduced in \cite{Ozawa1992}. It is based on the new variable
\begin{equation}\label{eq-phi}
\ph = \partial_t \psi,
\end{equation}
for which we get the equation
\begin{equation}\label{eq-phi-eq}
\iu \partial_t \ph = -\Delta \ph + u \ph + \dot{u} \psi.
\end{equation}
The original variable $\psi$ can be recovered from $\ph$ in two equivalent ways. On the one hand, we have
\begin{equation}\label{eq-EfromF-cont}
\psi(t) = \psi(t_0) + \int_{t_0}^t \ph(t')\,\drm t'.
\end{equation}
On the other hand, we get the Poisson equation $-\Delta \psi = \iu \ph - u\psi$ from the differential equation for $\psi$ in \eqref{eq-zak}, which yields 
\begin{equation}\label{eq-EfromF2-cont}
\psi = \klabig{-\Delta+1}^{-1} \klabig{\iu \ph + \psi - u \psi  },
\end{equation}
where $\psi$ on the right-hand side can be computed from \eqref{eq-EfromF-cont}.
For $\psi$ computed from the integral \eqref{eq-EfromF-cont}, we write $\psi_{I}$ in the following, and for $\psi$ computed from the Poisson equation \eqref{eq-EfromF2-cont} with $\psi_{I}$ on the right-hand side, we write $\psi_{P}$.

We end up with the system
\begin{subequations}\label{eq-zak-new}
\begin{equation}\label{eq-zak-new-1}\begin{split}
\iu \partial_t \ph &= -\Delta \ph + u \ph + \dot{u} \psi_I ,\\
\partial_{t} u &= \dot{u},\\
\partial_t \dot{u} &= \Delta u + \Delta \abs{\psi_P}^2,
\end{split}\end{equation}
where the notations
\begin{equation}\label{eq-zak-new-2}\begin{split}
\psi_I(t) &= \psi(t_0) + \int_{t_0}^t \ph(t')\,\drm t',\\
\psi_P &= \klabig{-\Delta+1}^{-1} \klabig{\iu \ph + \psi_I - u \psi_I }
\end{split}\end{equation}
\end{subequations}
are used.
In the new variables, we see that we gain two spatial derivatives of $\psi=\psi_P$ (in comparison to $\ph=\partial_t\psi$) as needed in the wave equation of \eqref{eq-zak-orig}, see the discussion in the introduction. 

For later use, we note that, under the regularity assumption \eqref{eq-regularity} of Theorem \ref{thm-semi}, we have for the new variable $\ph = \partial_t \psi = \iu \Delta \psi - \iu u \psi$ the bound 
\begin{equation}\label{eq-regularity-phi}
\norm{\ph(t)}_{s+\si} \le C M
\end{equation}
by the algebra property \eqref{eq-algebra}.

\subsection{Transformation of the semi-discretization in space}

It is straightforward to extend the described transformation of \cite{Ozawa1992} to the semi-dis\-cre\-ti\-zation in space \eqref{eq-zakK}. With the new variable
\begin{equation}\label{eq-phiK}
\ph^K = \partial_t \psi^K
\end{equation}
we get the system
\begin{subequations}\label{eq-zakK-new}
\begin{equation}\label{eq-zakK-new-1}\begin{split}
\iu \partial_t \ph^K &= -\Delta \ph^K + \Ic \klabig{u^K \ph^K + \dot{u}^K \psi_I^K} ,\\
\partial_{t} u^K &= \dot{u}^K,\\
\partial_t \dot{u}^K &= \Delta u^K + \Delta \Ic \klabig{ \abs{\psi_P^K}^2},
\end{split}\end{equation}
where the notations
\begin{equation}\label{eq-zakK-new-2}\begin{split}
\psi_I^K(t) &= \psi^K(t_0) + \int_{t_0}^t \ph^K(t')\,\drm t',\\
\psi_P^K &= \klabig{-\Delta+1}^{-1} \Ic \klabig{\iu \ph^K + \psi_I^K - u^K \psi_I^K }
\end{split}\end{equation}
\end{subequations}
are used. Note that $\psi_I^K=\psi_P^K=\psi^K$. The initial values for this system are
\[
\ph^K(t_0) = \Ic \klabig{\ph(t_0)}, \qquad u^K(t_0) = \Ic \klabig{u(t_0)}, \qquad \dot{u}^K(t_0) = \Ic \klabig{\dot{u}(t_0)}.
\]

\subsection{Transformation of the full discretization}

In this section, we perform the transformation of \cite{Ozawa1992} on a fully discrete level for the numerical method \eqref{eq-split}. This shows how the Lie--Trotter splitting \eqref{eq-split} can be interpreted as a discretization of the Zakharov system in the new variables \eqref{eq-zak-new}. 

The discrete new variable that we introduce is 
\[
\ph_{n+1}^K = \frac{\psi_{n+1}^K-\psi_n^K}{\tau}, \qquad n=0,1,2,\ldots .
\]
This is a \emph{time-discrete} version of \eqref{eq-phi} and \eqref{eq-phiK}. 
For this new variable, we get from \eqref{eq-split-1} the recursion
\[
\ph_{n+1}^K = \e^{-\iu\tau\Om^2} \Ic \klabig{ \e^{-\iu \tau v_{n}^K} \ph_n^K } - \iu \e^{-\iu\tau\Om^2} \Ic \klaBigg{\klabigg{\frac{\e^{-\iu\tau v_{n+1}^K}-\e^{-\iu\tau v_{n}^K}}{-\iu\tau}} \psi_n^K},
\]
which is a discrete analogon of \eqref{eq-phi-eq}.
Writing $w_n^K = (v_{n+1}^K - v_{n}^K)/\tau$, 
we get
\[
\ph_{n+1}^K = \e^{-\iu\tau\Om^2} \Ic \klaBig{\e^{-\iu \tau v_{n}^K} \klaBig{\ph_n^K  -\iu\tau w_n^K \phi\klabig{-\iu\tau^2 w_n^K} \psi_n^K} }
\]
with the function $\phi$ of \eqref{eq-phifunc}.
We note that
\[
w_n^K = \sinc\klabig{\sfrac12 \tau\Om}^2 \dot{u}_n^K + \klabig{\sinc\kla{\tau\Om}-1} \Ic \klaBig{ \ReT \klabig{(\psi_n^K+\psi_{n-1}^K) \overline{\ph}_n^K} }
\]
by \eqref{eq-vn-1} and \eqref{eq-vn-2}.

The original variable $\psi$ can be recovered from the new variable $\ph$ in two equivalent ways. On the one hand, we have in analogy to \eqref{eq-EfromF-cont}
\[
\psi_{n}^K = \psi_{n-1}^K + \tau \ph_{n}^K = \dots = \psi_0^K + \tau \klabig{ \ph_1^K + \dots + \ph_{n}^K }.
\]
On the other hand, using $\psi_{n-1}^K = \Ic(\e^{\iu\tau v_{n}^K} \e^{\iu\tau\Om^2}\psi_{n}^K)$ in $\tau\ph_n^K=\psi_n^K-\psi_{n-1}^K$, we get
\[
\klabigg{\frac{\e^{\iu\tau\Om^2}-1}{\iu\tau}} \psi_{n}^K = \iu\ph_{n}^K - \Ic \klaBigg{ \klabigg{\frac{\e^{\iu\tau v_{n}^K}-1}{\iu\tau}} \e^{\iu\tau\Om^2} \psi_{n}^K }.
\]
Under the CFL condition \eqref{eq-cfl}, the matrix $\e^{\iu\tau\Om^2}-1$ is invertible. This yields a discrete analogon of \eqref{eq-EfromF2-cont}:
\[
\psi_{n}^K = \klabig{\Om^2\phi\klabig{\iu\tau\Om^2}+1}^{-1} \Ic \klaBig{\iu\ph_{n}^K + \psi_n^K - v_n^K\phi\klabig{\iu\tau v_{n}^K} \e^{\iu\tau\Om^2} \psi_{n}^K }
\]
with $\phi$ from \eqref{eq-phifunc}.

We end up with (for $n\ge 1$)
\begin{subequations}\label{eq-split-new}
\begin{equation}\label{eq-split-new-1}\begin{split}
\ph_{n+1}^K &= \e^{-\iu\tau\Om^2} \Ic \klaBig{\e^{-\iu \tau v_{n}^K} \klabig{\ph_n^K  -\iu\tau w_n^K \phi\klabig{-\iu\tau^2 w_n^K} \psi_{I,n}^K} }, \\
\begin{pmatrix} u_{n+1}^K\\ \dot{u}_{n+1}^K \end{pmatrix} &= R(\tau) \begin{pmatrix} u_{n}^K\\ \dot{u}_{n}^K \end{pmatrix}
+ \klabig{R(\tau)-1}  \begin{pmatrix} \Ic \kla{\abs{\psi_{P,n}^K}^2} \\ 0 \end{pmatrix},
\end{split}\end{equation}
where the notations
\begin{equation}\label{eq-split-new-vw}\begin{split}
v_{n}^K &= \sinc\kla{\tau\Om} u_{n}^K - \sfrac12 \tau \sinc\klabig{\sfrac12 \tau\Om}^2 \dot{u}_{n}^K + \klabig{\sinc\kla{\tau\Om}-1} \Ic \klabig{\abs{\psi_{I,n-1}^K}^2},\\
w_{n}^K &= \sinc\klabig{\sfrac12 \tau\Om}^2 \dot{u}_n^K + \klabig{\sinc\kla{\tau\Om}-1} \Ic \klaBig{ \ReT \klabig{(\psi_{I,n}^K+\psi_{I,n-1}^K) \overline{\ph}_n^K} }
\end{split}\end{equation}
and
\begin{equation}\label{eq-split-new-psi}\begin{split}
\psi_{I,n}^K &= \psi_0^K + \tau \klabig{\ph_1^K+\dots+\ph_n^K},\\
\psi_{P,n}^K &= \klabig{\Om^2\phi\klabig{\iu\tau\Om^2}+1}^{-1} \Ic \klaBig{\iu\ph_{n}^K + \psi_{I,n}^K - v_n^K\phi\klabig{\iu\tau v_{n}^K} \e^{\iu\tau\Om^2} \psi_{I,n}^K }
\end{split}\end{equation}
are used. Note again that $\psi_{I,n}^K=\psi_{P,n}^K=\psi_n^K$. 
The starting values for this two-term recursion are
\begin{equation}\label{eq-split-new-start}\begin{split}
&u_0^K = \Ic\klabig{u(t_0)}, \qquad \dot{u}_0^K= \Ic\klabig{\dot{u}(t_0)}, \qquad \psi_{P,0}^K=\psi_0^K=\Ic\klabig{\psi(t_0)},\\
&\ph_1^K = \frac{\psi_1^K-\psi_0^K}{\tau} \qquad\text{with}\qquad \psi_1^K = \e^{-\iu\tau\Om^2} \Ic\klabig{\e^{-\iu \tau v_{1}^K} \psi_0^K}.
\end{split}\end{equation}
\end{subequations}
The formulas for the numerical method in the new variables are much longer than in the original variables and certainly not suited for an implementation of the method. Nevertheless, they are very useful for the analysis: it turns out that they allow us to gain two spatial derivatives for the numerical solution $\psi_n^K=\psi_{P,n}^K$ (in comparison to $\ph_n^K$) thanks to the factor $\kla{\Om^2\phi\klabig{\iu\tau\Om^2}+1}^{-1}$. This property is crucial for our analysis. It is proven in the following lemma under the CFL condition \eqref{eq-cfl}. Under this CFL condition, the operator $\Om^2=-\Delta$ acts only on trigonometric polynomials of fixed degree and is thus bounded. 

\begin{lemma}\label{lemma-poisson}
Let $s\ge 0$ and assume that the CFL condition \eqref{eq-cfl} holds. We then have, for trigonometric polynomials $v(x)=\sum_{j\in\mathcal{K}} v_j \, \e^{\iu j\cdot x}$ of degree $K$,
\begin{align*}
\normbig{\klabig{\Om^2\phi\klabig{\iu\tau\Om^2}+1}^{-1} v}_{s+2} &\le C \norm{v}_s,\\
\normbig{\klabig{\Om^2+1}^{-1} v - \klabig{\Om^2\phi\klabig{\iu\tau\Om^2}+1}^{-1} v}_s &\le C \tau \norm{v}_s.
\end{align*}
\end{lemma}
\begin{proof}
(a) We prove the first estimate. The $j$th Fourier coefficient of the function whose norm has to be estimated is given by
\begin{equation}\label{eq-aux-fouriercoeff1}
\frac{\iu\tau }{ \e^{\iu\tau \abs{j}^2}-1 + \iu\tau} \,\, v_j , \qquad j\in\mathcal{K} .
\end{equation}
For the expression in the denominator, we note that 
\[
e^{\iu\tau \abs{j}^2} - 1 + \iu\tau = -2\sin\klabig{\sfrac12 \tau \abs{j}^2}^2 + \iu \klabig{\sin(\tau \abs{j}^2) + \tau}. 
\]
Using $\sin(\xi)\ge \frac2\pi \xi$ for $0\le \xi\le \frac12 \pi$, we thus have for this denominator 
\begin{subequations}\label{eq-aux-denom}
\begin{equation}
\absbig{e^{\iu\tau \abs{j}^2} - 1 + \iu\tau} \ge \sin(\tau \abs{j}^2) + \tau \ge \sfrac2\pi \tau \max\klabig{\abs{j},1}^2 , \qquad 0\le \tau \abs{j}^2 \le \sfrac12 \pi .
\end{equation}
Using $\sin(\frac12 \xi)\ge \xi \sin(\frac12 c)/c$ and $\sin(\frac12\xi)\ge \min(\sin(\frac12 c),\sin(\frac14\pi))$ for $\frac12 \pi\le \xi\le c$ with $c<2\pi$ from the CFL condition \eqref{eq-cfl}, we get
\begin{equation}
\absbig{e^{\iu\tau \abs{j}^2} - 1 + \iu\tau} \ge 2 \sin\klabig{\sfrac12 \tau \abs{j}^2}^2 \ge C \tau \max\klabig{\abs{j},1}^2, \qquad \sfrac12 \pi \le \tau \abs{j}^2 \le c
\end{equation}
\end{subequations}
with $C$ depending on $c<2\pi$.
Using the estimates \eqref{eq-aux-denom} for the denominator in \eqref{eq-aux-fouriercoeff1}, we see that the absolute value of the above $j$th Fourier coefficient \eqref{eq-aux-fouriercoeff1} is bounded by $C \max(\abs{j},1)^{-2} \abs{v_j}$, if the CFL condition \eqref{eq-cfl} holds. This yields the first estimate of the lemma.

(b) For the second estimate of the lemma, we proceed similarly as in (a). The $j$th Fourier coefficient of the function whose norm has to be estimated is now given by
\begin{equation}\label{eq-aux-fouriercoeff2}
\frac{\e^{\iu\tau \abs{j}^2} - 1 - \iu\tau \abs{j}^2}{(\abs{j}^2+1)(e^{\iu\tau \abs{j}^2} - 1 + \iu\tau)} \, v_j , \qquad j\in\mathcal{K} . 
\end{equation}
Using the estimate \eqref{eq-aux-denom} in the denominator and $\abs{\e^{\iu \xi}-1-\iu \xi} = \abs{ \int_0^{\xi} \int_0^{\eta} \iu^2 \e^{\iu \nu} \,\drm \nu \,\drm \eta} \le \abs{\xi}^2$ in the numerator, we see that the absolute value of \eqref{eq-aux-fouriercoeff2} is bounded by $C \tau \abs{v_j}$. This yields the second claimed estimate.
\end{proof}

\section{Error analysis of the semi-discretization in space}\label{sec-proof-semi}

In this section, we give the proof of the error bound of Theorem \ref{thm-semi} for the Fourier collocation in space \eqref{eq-zakK}. We do so by interpreting the semi-discretization in space in the new variables \eqref{eq-zakK-new} as a discretization of the Zakharov system in the new variables \eqref{eq-zak-new}, and we study the error of this discretization. Translating the result back to the original variables then leads to the error bound of Theorem \ref{thm-semi}.

More precisely, the proof is organized as follows. We start in Section \ref{subsec-semi-voc} below with the variation-of-constants formulas for the exact solution in new variables and the solution of the spatial semi-discretization in new variables. In Section \ref{subsec-semi-bounds}, we then derive bounds on all the terms that appear in the difference of the two variation-of-constants formulas. They are derived under an additional regularity assumption on the solution of the semi-discretization in space. The estimates include in particular error bounds for the variables $\psi_I$ and $\psi_P$ and bounds on the operator $R$ of \eqref{eq-R}. In the final Section~\ref{subsec-semi-proof}, the bounds of Section \ref{subsec-semi-bounds} are used to derive Theorem \ref{thm-semi} with a Gronwall inequality. The additional regularity assumption of Section \ref{subsec-semi-bounds} on the solution of the semi-discretization in space is justified with a bootstrap argument. 

Throughout this section, we let $s>\frac{d}2$ and $\si>0$ as in Theorem \ref{thm-semi}, and we denote by $C$ a generic constant that may depend on $s$ and $\si$ and in addition on the dimension $d$, on the final time $T$ and on the constant $M$ of the regularity assumption \eqref{eq-regularity} of Theorem \ref{thm-semi}. 

\subsection{Variation-of-constants formula}\label{subsec-semi-voc}

The proof of Theorem \ref{thm-semi} is based on the variation-of-constants formula applied to the difference of  the spatial semi-discretization in the new variables \eqref{eq-zakK-new} and the exact solution in the new variables \eqref{eq-zak-new}. This reads ($-\Delta=\Om^2$)
\begin{equation}\label{eq-voc-semi}\begin{split}
\ph(t)-\ph^K(t) &= \e^{-\iu(t-t_0)\Om^2}\klabig{\ph(t_0)-\ph^K(t_0)} - \iu \int_{t_0}^t \e^{-\iu(t-t')\Om^2} \theta(t_0+t') \,\drm t',\\ 
\begin{pmatrix} u(t)-u^K(t) \\ \dot{u}(t)-\dot{u}^K(t) \end{pmatrix} &= R(t-t_0) \begin{pmatrix} u(t_0)-u^K(t_0)\\ \dot{u}(t_0)-\dot{u}^K(t_0) \end{pmatrix} 
- \int_{t_0}^t R(t-t')  \begin{pmatrix} 0 \\ \Om^2 \vartheta(t_0+t') \end{pmatrix} \, \drm t'
\end{split}\end{equation}
with
\[
\theta = \klabig{u\ph+\dot{u}\psi_I} - \Ic \klabig{u^K \ph^K + \dot{u}^K \psi_I^K} \myand
\vartheta = \abs{\psi_P}^2 - \Ic \klabig{ \abs{\psi_P^K}^2}.
\]
In Section \ref{subsec-semi-bounds} below, we bound the terms on the right-hand side of \eqref{eq-voc-semi}. The proof of Theorem \ref{thm-semi} is then given in Section \ref{subsec-semi-proof}.

\subsection{Bounds on the terms in the variation-of-constants formula}\label{subsec-semi-bounds}

We denote by
\begin{equation}\label{eq-errormax-semi}
e(t) = \sup_{t_0\le t' \le t} \klaBig{ \normbig{\ph(t')-\ph^K(t')}_s + \normvbig{\klabig{u(t'),\dot{u}(t')} - \klabig{u^K(t'),\dot{u}^K(t')}}_{s} }
\end{equation}
the maximal error in $\ph$, $u$ and $\dot{u}$ until time $t$. 

In the following lemmas, we will assume regularity \eqref{eq-regularity} (see also \eqref{eq-regularity-phi}) of the exact solution. In addition to that, we will also assume regularity of the semi-discrete solution \eqref{eq-zakK}:  
\begin{equation}\label{eq-regularityK}
\normbig{\psi^K(t)}_{s+2} + \normvbig{\klabig{u^K(t),\dot{u}^K(t)}}_{s} \le 2 M \myfor 0\le t-t_0\le T
\end{equation}
with $M$ from the regularity assumption \eqref{eq-regularity}. The bound \eqref{eq-regularityK} will be justified in final proof of Theorem \ref{thm-semi} below by a bootstrap argument (note that the norms in \eqref{eq-regularityK} are those in which the error bound of Theorem~\ref{thm-semi} has to be shown). 

\begin{lemma}[Error in $\psi_I$]\label{lemma-error-psiI}
Under the assumption \eqref{eq-regularity}, we have, for $0\le t\le T$,
\[
\normbig{\psi_I(t) - \psi_I^K(t) }_{s} \le C \klabig{K^{-2-\si} + e(t) }.
\]
\end{lemma}
\begin{proof}
This follows from the definitions of $\psi_I$ in \eqref{eq-zak-new-2}, $\psi_I^K$ in \eqref{eq-zakK-new-2} and $e$ in \eqref{eq-errormax-semi}, and from Lemma \ref{lemma-inter} applied to the interpolation error $\psi(t_0)-\psi^K(t_0)=\psi(t_0)-\Ic(\psi(t_0))$.
\end{proof}

\begin{lemma}[Error in $\psi_P$]\label{lemma-error-psiP}
Under the assumption \eqref{eq-regularity}, we have, for $0\le t\le T$,
\[
\normbig{\psi_P(t) - \psi_P^K(t) }_{s+2} \le C \klabig{K^{-\si} + e(t) }.
\]
\end{lemma}
\begin{proof}
We write
\begin{multline*}
\klabig{-\Delta+1} \klabig{\psi_P - \psi_P^K} = (1-\Ic) \klabig{\iu \ph + \psi_I - u \psi_I } + \Ic \klaBig{ \klabig{\iu\ph-\iu\ph^K}\\ + \klabig{\psi_I-\psi_I^K} + \klabig{u-u^K}\psi_I + u^K\klabig{\psi_I- \psi_I^K} }.
\end{multline*}
By the algebra property \eqref{eq-algebra} and the bounds \eqref{eq-regularity} and \eqref{eq-regularity-phi}, we have $\norm{\iu\ph+\psi_I-u\psi_I}_{s+\si}\le C$. The interpolation error in the above decomposition can thus be estimated with Lemma~\ref{lemma-inter}:
\[
\normbig{(1-\Ic) \klabig{\iu \ph + \psi_I - u \psi_I } }_{s} \le C K^{-\si}.
\]
The statement of the lemma then follows from Lemma \ref{lemma-inter}, the bounds \eqref{eq-regularity} on $\psi=\psi_I$ and \eqref{eq-regularityK} on $u^K$, the algebra property \eqref{eq-algebra}, the definition \eqref{eq-errormax-semi} of $e(t)$ and Lemma \ref{lemma-error-psiI} on $\psi_I-\psi_I^K$ applied to the second term in the above decomposition.
\end{proof}

\begin{lemma}[Bound of $\theta$]\label{lemma-bound-theta}
Under the regularity assumptions \eqref{eq-regularity} and \eqref{eq-regularityK}, we have, for $0\le t\le T$,
\[
\norm{\theta(t)}_{s} \le C \klabig{ K^{-\si} + e(t) }.
\]
\end{lemma}
\begin{proof}
We write 
\begin{multline*}
\theta = (1-\Ic) \klabig{u\ph+\dot{u}\psi_I} + \Ic \klaBig{ \klabig{u-u^K}\ph + u^K\klabig{\ph-\ph^K}\\ + \klabig{\dot{u}-\dot{u}^K}\psi_I + \dot{u}^K\klabig{\psi_I-\psi_I^K}}.
\end{multline*}
As in the proof of Lemma \ref{lemma-error-psiP}, we get 
\[
\normbig{(1-\Ic) \klabig{u\ph+\dot{u}\psi_I}}_s \le C K^{-\si}
\]
for the interpolation error in the above decomposition, and we then apply Lemma \ref{lemma-inter}, the algebra property \eqref{eq-algebra}, the bounds \eqref{eq-regularity}, \eqref{eq-regularity-phi} and \eqref{eq-regularityK}, the definition \eqref{eq-errormax-semi} of $e(t)$ and Lemma \ref{lemma-error-psiI} on $\psi_I-\psi_I^K$ to the second term in the above decomposition.
\end{proof}

\begin{lemma}[Bound of $\vartheta$]\label{lemma-bound-vartheta}
Under the regularity assumptions \eqref{eq-regularity} and \eqref{eq-regularityK}, we have, for $0\le t\le T$,
\[
\norm{\vartheta(t)}_{s+2} \le C \klabig{ K^{-\si} + e(t) }.
\]
\end{lemma}
\begin{proof}
We write 
\[
\vartheta = (1-\Ic) \klabig{\abs{\psi_P}^2} +  \Ic \klaBig{ \ReT \klaBig{ \klabig{\psi_P - \psi_P^K} \klabig{\overline{\psi_P+\psi_P^K}} }}.
\]
The claimed estimate then follows from Lemma \ref{lemma-inter}, the bounds \eqref{eq-regularity} and \eqref{eq-regularityK} on $\psi=\psi_p$ and $\psi^K=\psi_P^K$, respectively, the algebra property \eqref{eq-algebra} and Lemma \ref{lemma-error-psiP} on $\psi_P-\psi_P^K$.
\end{proof}

In addition to the above bounds on $\theta$ and $\vartheta$, we need the following bound on the matrix $R$ of \eqref{eq-R}, which also appears in the variation-of-constants formula \eqref{eq-voc-semi}. The bound states that this matrix almost preserves the norm $\normv{\cdot}_s$.

\begin{lemma}\label{lemma-bound-R}
We have
\[
\normvbig{ R(t) \klabig{\begin{smallmatrix} v\\ \dot{v} \end{smallmatrix}} }_s \le (1+\abs{t}) \normvbig{ \klabig{\begin{smallmatrix} v\\ \dot{v} \end{smallmatrix}} }_s, \qquad \klabig{\begin{smallmatrix} v\\ \dot{v} \end{smallmatrix}} \in H^{s+1}\times H^s.
\]
\end{lemma}
\begin{proof}
Let $v(x)=\sum_{j\in\mathbb{Z}^d} v_j \, \e^{\iu j\cdot x}$ and $\dot{v}(x)=\sum_{j\in\mathbb{Z}^d} v_j \, \e^{\iu j\cdot x}$.
Note that, by the definition of the norm $\normv{\cdot}_s$,
\begin{multline*}
\normvbig{ R(t) \klabig{\begin{smallmatrix} v\\ \dot{v} \end{smallmatrix}} }_s = \klaBig{\normbig{\Om^{s+1} \klabig{\cos(t\Om)v + t\sinc(t\Om)\dot{v}}}_0^2 + \absbig{v_0 + t\dot{v}_0}^2\\
 + \normbig{\Om^{s}\klabig{-\Om\sin(t\Om)v + \cos(t\Om)\dot{v}}}_0^2 + \abs{\dot{v}_0}^2}^{1/2}.
\end{multline*}
By the triangle inequality, we can estimate this quantity by the sum of 
\begin{multline*}
\klaBig{\normbig{\Om^{s+1}\cos(t\Om)v + \Om^s\sin(t\Om)\dot{v}}_0^2 + \abs{v_0}^2\\
 + \normbig{-\Om^{s+1}\sin(t\Om)v + \Om^s\cos(t\Om)\dot{v}}_0^2 + \abs{\dot{v}_0}^2}^{1/2} = \normvbig{\klabig{\begin{smallmatrix} v\\ \dot{v} \end{smallmatrix}} }_s
\end{multline*}
and
\[
\klabig{\abs{t \dot{v}_0}^2}^{1/2} \le \abs{t} \cdot \normvbig{\klabig{\begin{smallmatrix} v\\ \dot{v} \end{smallmatrix}} }_0 \le \abs{t} \cdot \normvbig{\klabig{\begin{smallmatrix} v\\ \dot{v} \end{smallmatrix}} }_s.
\]
This yields the claimed estimate.
\end{proof}

\subsection{Proof of Theorem \ref{thm-semi}}\label{subsec-semi-proof}

We start from the variation-of-constants formulas \eqref{eq-voc-semi}. Taking norms in this formulas and using Lemmas \ref{lemma-bound-theta}, \ref{lemma-bound-vartheta} and \ref{lemma-bound-R} on $\theta$, $\vartheta$ and $R$, respectively, yields for the error $e$ of~\eqref{eq-errormax-semi}
\[
e(t) \le (1+\abs{t-t_0}) e(t_0) + C \int_{t_0}^t (1+\abs{t-t'}) \klabig{K^{-\si}+e(t')} \,\drm t'
\]
as long as the semi-discrete solution satisfies the bounds \eqref{eq-regularityK}. 
Using Lemma \ref{lemma-inter} to estimate $e(t_0)$ and the Gronwall lemma, this implies
\[
e(t) \le C K^{-\si}
\]
for $0\le t-t_0\le T$ and as long as \eqref{eq-regularityK} holds.
This estimate contains the error bounds for $u$ and $\dot{u}$ of Theorem \ref{thm-semi}. To get also the error bound for $\psi=\psi_P$ of Theorem \ref{thm-semi}, we use Lemma \ref{lemma-error-psiP}. 

To complete the proof, we still have to justify that \eqref{eq-regularityK} holds. As $\si>0$, the above error bound and the regularity assumption~\eqref{eq-regularity} on the exact solution show that assumption \eqref{eq-regularityK} even holds with the better constant $\tfrac32 M$ instead of $2M$, provided that~$K$ is sufficiently large (such that $CK^{-\sigma}\le\tfrac12 M$, where $C$ is the constant of the error bound). By a bootstrap argument (see, e.g., \cite[Section 1.3]{Tao2006}), the proof of Theorem~\ref{thm-semi} is thus complete.

\section{Error analysis of the full discretization}\label{sec-proof}

In this section, we give the proof of the global error bound for the Lie-Trotter splitting~\eqref{eq-split} stated in Theorem \ref{thm-main}. As in the previous section, we interpret the splitting method in the new variables \eqref{eq-split-new} as a discretization of the Zakharov system in the new variables \eqref{eq-zak-new}, and we study the error of this discretization. 
In view of Theorem \ref{thm-semi} on the error of the semi-discretization in space, we only have to consider here the temporal errors (in the new variables of Section \ref{sec-transformation})
\begin{equation}\label{eq-error-temporal}
\ph^K(t_{n})-\ph_{n}^K, \qquad u^K(t_{n})-u_{n}^K, \qquad \dot{u}^K(t_{n})-\dot{u}_{n}^K,
\end{equation}
with the space-discrete solutions $\ph^K$, $u^K$ and $\dot{u}^K$ of \eqref{eq-zakK-new}. 

Throughout this section, we let $s>\frac{d}2$ and $\si\ge 2$ as in Theorem \ref{thm-main}. As in the previous section, we denote by $C$ a generic constant that may depend on $s$, $\si$, the dimension $d$, the final time $T$ and on the constant $M$ of the regularity assumption~\eqref{eq-regularity}, and now in addition on the constant $c$ of the CFL condition \eqref{eq-cfl}.

\subsection{Lady Windermere's fan and outline of the proof}\label{subsec-windermere}

We decompose the temporal errors \eqref{eq-error-temporal} after $n+1$ time steps as
\begin{equation}\label{eq-decom}
\ph^K(t_{n+1})-\ph_{n+1}^K = \klaBig{\ph^K(t_{n+1})-\widehat{\ph}_{n+1}^K} + \klaBig{\widehat{\ph}_{n+1}^K - \ph_{n+1}^K}
\end{equation}
and similarly for $u(t_{n+1})-u_{n+1}^K$ and $\dot{u}(t_{n+1})-\dot{u}_{n+1}^K$. In this decomposition, 
\[
\widehat{\ph}_{n+1}^K, \qquad \widehat{u}_{n+1}^K, \qquad \widehat{\dot{u}}_{n+1}^K
\]
are (essentially) numerical solutions at time $t_{n+1}=t_0+(n+1)\tau$ when starting with 
the exact solution to the spatial semi-discretization \eqref{eq-zakK-new} as initial values at time $t_n=t_0+n\tau$. More precisely,
\begin{subequations}\label{eq-split-hat}
\begin{equation}\label{eq-split-hat-1}\begin{split}
\widehat{\ph}_{n+1}^K &= \e^{-\iu\tau\Om^2} \Ic \klaBig{\e^{-\iu \tau \widehat{v}_{n}^K} \klabig{\ph^K(t_n)  -\iu\tau \widehat{w}_n^K \phi\klabig{-\iu\tau^2 \widehat{w}_n^K} \widehat{\psi}_{I,n}^K} },\\
\begin{pmatrix} \widehat{u}_{n+1}^K\\ \widehat{\dot{u}}_{n+1}^K \end{pmatrix} &= R(\tau) \begin{pmatrix} u^K(t_{n})\\ \dot{u}^K(t_{n}) \end{pmatrix} 
+ \klabig{R(\tau)-1}  \begin{pmatrix} \Ic \klabig{\abs{\widehat{\psi}_{P,n}^K}^2} \\ 0 \end{pmatrix}
\end{split}\end{equation}
with
\begin{equation}\label{eq-split-hat-vw}\begin{split}
\widehat{v}_{n}^K &= \sinc\kla{\tau\Om} u^K(t_{n}) - \sfrac12 \tau \sinc\klabig{\sfrac12 \tau\Om}^2 \dot{u}^K(t_{n}) + \klabig{\sinc\kla{\tau\Om}-1} \Ic \klabig{\abs{\widehat{\psi}_{I,n-1}^K}^2},\\
\widehat{w}_{n}^K &= \sinc\klabig{\sfrac12 \tau\Om}^2 \dot{u}^K(t_n) + \klabig{\sinc\kla{\tau\Om}-1} \Ic \klaBig{ \ReT \klabig{(\widehat{\psi}_{I,n}^K+\widehat{\psi}_{I,n-1}^K) \overline{\ph^K}(t_n)} } 
\end{split}\end{equation}
and
\begin{equation}\label{eq-split-hat-psi}\begin{split}
\widehat{\psi}_{I,n}^K &= \psi_0^K + \tau \klabig{\ph^K(t_1)+\dots+\ph^K(t_n)},\\
\widehat{\psi}_{P,n}^K &= \klabig{\Om^2\phi\klabig{\iu\tau\Om^2}+1}^{-1} \Ic \klaBig{\iu\ph^K(t_{n}) + \widehat{\psi}_{I,n}^K - \widehat{v}_n^K\phi\klabig{\iu\tau \widehat{v}_{n}^K} \e^{\iu\tau\Om^2} \widehat{\psi}_{I,n}^K },
\end{split}\end{equation}
and with 
\begin{equation}\label{eq-split-hat-start}
  \widehat{\psi}_{P,0}^K=\psi_{P,0}^K=\psi_0^K, \qquad \widehat{\ph}_1^K = \ph_1^K.
\end{equation}
\end{subequations}
Note that $\widehat{\psi}_{I,n}^K$ and $\widehat{\psi}_{P,n}^K$ are here not necessarily identical (although $\psi_{I,n}^K=\psi_{P,n}^K$).

The above decomposition \eqref{eq-decom} is at the heart of the proof of Theorem \ref{thm-main} with Lady Windermere's fan.
In this decomposition, the differences 
\begin{equation}\label{eq-errors-stability}
\widehat{\ph}_{n+1}^K-\ph_{n+1}^K, \qquad \widehat{u}_{n+1}^K-u_{n+1}^K, \qquad \widehat{\dot{u}}_{n+1}^K-\dot{u}_{n+1}^K
\end{equation}
describe the propagation of the global error after $n$ time steps by the numerical method, and the differences
\begin{equation}\label{eq-errors-local}
\ph^K(t_{n+1})-\widehat{\ph}_{n+1}^K,\qquad u(t_{n+1})^K-\widehat{u}_{n+1}^K, \qquad \dot{u}^K(t_{n+1})-\widehat{\dot{u}}_{n+1}^K
\end{equation}
are local errors of the numerical method.

The crucial ingredients of the proof of Theorem \ref{thm-main} are bounds of these differences \eqref{eq-errors-stability} and \eqref{eq-errors-local}. Before deriving these bounds, however, we prove in Section \ref{subsec-full-bounds} below regularity properties of the solution to the spatial semi-discretization in new variables~\eqref{eq-zakK-new} and the intermediate solution \eqref{eq-split-hat}. These properties are then used in Section \ref{subsec-stability} below in combination with an additional regularity assumption on the fully discrete solution to study stability of the method by estimating the differences~\eqref{eq-errors-stability}. As the method in the new variables involves several auxiliary variables, we study stability in one of these variables after the other. In Section \ref{subsec-localerror} below, we then derive bounds on the local error~\eqref{eq-errors-local}. Again, we do this first for all auxiliary variables and then for the main variables appearing in \eqref{eq-errors-local}. In the final Section \ref{subsec-proof-main}, we put stability and local error bounds together to prove Theorem \ref{thm-main}, thereby ensuring the additional regularity assumption of Section~\ref{subsec-stability} by an inductive argument.

\subsection{Bounds on the spatially discrete and the intermediate solution}\label{subsec-full-bounds}

Before estimating error terms, we collect the following bounds for the solution of the spatial semi-discretization \eqref{eq-zakK-new} and for the intermediate solution defined in \eqref{eq-split-hat}. Both of them appear in the decomposition \eqref{eq-decom} of the error.

\begin{lemma}[Bound on the solution of the spatial semi-discretization]
Under the regularity assumption \eqref{eq-regularity}, we have, for $0\le t-t_0\le T$,
\begin{equation}\label{eq-regularityK-high}
\normbig{\psi^K(t)}_{s+4} + \normvbig{\klabig{u^K(t),\dot{u}^K(t)}}_{s+2} + \normbig{\ph^K(t)}_{s+2} \le C .
\end{equation}
\end{lemma}
\begin{proof}
We only prove the estimate of $\psi^K$. The estimates of $u^K$ and $\dot{u}^K$ are obtained similarly, and the estimate of $\ph^K$ follows from these estimates as in \eqref{eq-regularity-phi}.

We decompose
\[
\psi^K = \klaBig{\psi^K(t)-\Ic\klabig{\psi(t)}} + \Ic\klabig{\psi(t)}.
\]
The second term can be estimated with Lemma \ref{lemma-inter} and \eqref{eq-regularity} (note that $\si\ge 2$).
For the first term, we use that
\begin{align*}
&\normbig{\psi^K(t)-\Ic\klabig{\psi(t)}}_{s+4} \le d K^2 \normbig{\psi^K(t)-\Ic\klabig{\psi(t)}}_{s+2}\\
 &\qquad\qquad\qquad \le d K^2 \normbig{\psi^K(t)-\psi(t)}_{s+2} + d K^2 \normbig{\Ic\klabig{\psi(t)} - \psi(t)}_{s+2} \le C
\end{align*}
by the inverse estimate $\norm{v}_{s+4}\le d K^2 \norm{v}_{s+2}$ for trigonometric polynomials $v$ of degree $K$, the error bound of Theorem \ref{thm-semi}, Lemma \ref{lemma-inter} and the bound \eqref{eq-regularity} (note again that $\si\ge 2$). This yields the claimed estimate of $\psi^K$.
\end{proof}

\begin{lemma}[Bound on the intermediate solution]\label{lemma-regularity-hat}
Under the regularity assumption \eqref{eq-regularity}, we have, for $\tau\le t_n-t_0\le T$,
\begin{equation}\label{eq-bound-hat-noncfl}
\normbig{\widehat{\psi}_{I,n}^K}_{s+2} + \normbig{\widehat{v}_n^K}_{s+2} + \normbig{\widehat{w}_n^K}_{s+2} + \normbig{\widehat{\ph}_{n+1}^K}_{s+2} \le C.
\end{equation}
Assuming in addition the CFL condition \eqref{eq-cfl}, we also have 
\begin{equation}\label{eq-bound-hat-cfl}
\normbig{\widehat{\psi}_{P,n}^K}_{s+4} + \normvbig{\klabig{\widehat{u}^K_{n},\widehat{\dot{u}}^K_{n}}}_{s+2} \le C.
\end{equation}
\end{lemma}
\begin{proof}
We first note that for $s'>\frac{d}2$
\begin{equation}\label{eq-exp}
\norm{\e^{v}}_{s'} \le \e^{C\norm{v}_{s'}}, \qquad 
\norm{\phi(v)}_{s'} \le \e^{C\norm{v}_{s'}}, \qquad v\in H^{s'}, 
\end{equation}
with $C$ depending on $d$ and $s'$, 
which follows by using the exponential series and the algebra property \eqref{eq-algebra}.
These properties together with the bounds \eqref{eq-regularityK-high} of the spatial semi-discretization, the algebra property \eqref{eq-algebra} and Lemma \ref{lemma-inter} imply the estimate~\eqref{eq-bound-hat-noncfl}. For the estimate \eqref{eq-bound-hat-cfl}, we use in addition Lemma \ref{lemma-poisson} on $(\Om^2\phi(\iu\tau\Om^2)+1)^{-1}$ and Lemma \ref{lemma-bound-R} on the matrix $R$ (with $s+2$ instead of $s$).
\end{proof}

\subsection{Stability estimates}\label{subsec-stability}

We study the error propagation of the Lie-Trotter splitting in the new variables \eqref{eq-split-new} by estimating the differences \eqref{eq-errors-stability}. Recall that the numerical method in the new variables is given by \eqref{eq-split-new} and that the intermediate solution appearing in these differences is given by \eqref{eq-split-hat}.

We denote for $n\ge 1$ by
\begin{equation}\label{eq-errormax}
e_n^K = \max_{j=1,\dots,n} \klaBig{ \normbig{\ph^K(t_j)-\ph_j^K}_s + \normvbig{\klabig{u^K(t_j),\dot{u}^K(t_j)} - \klabig{u_j^K , \dot{u}_j^K} }_{s} }
\end{equation}
the maximal error in $\ph$, $u$ and $\dot{u}$ of the discretization in time until time $t_n=t_0+n\tau$. For convenience, we set $e_0^K=0$ (recall that $\ph_0^K$ is not defined). 

In addition to the regularity assumption \eqref{eq-regularity} on the exact solution of Theorem \ref{thm-main}, we will assume in this section that the fully discrete numerical solution is bounded in the spaces in which the error bound of Theorem \ref{thm-main} is supposed to be shown:
\begin{equation}\label{eq-reg-numerical}
\normbig{\ph_n^K}_s + \normvbig{\klabig{u_n^K,\dot{u}_n^K}}_s \le 2 C M \myfor \tau\le t_n-t_0\le T
\end{equation}
with the constant $C$ of \eqref{eq-regularityK-high}. This assumption will be justified below in the final proof of Theorem \ref{thm-main}.
By the algebra property \eqref{eq-algebra} and Lemma \ref{lemma-inter}, this estimate implies in particular
\begin{equation}\label{eq-reg-numerical-vwpsiI}
\normbig{\psi_{I,n}^K}_s + \normbig{v_n^K}_s + \normbig{w_n^K}_s \le C \myfor \tau\le t_n-t_0\le T.
\end{equation}
Under the CFL condition \eqref{eq-cfl} and  using in addition Lemma \ref{lemma-poisson} and \eqref{eq-exp}, we also get
\begin{equation}\label{eq-reg-numerical-psiP}
\normbig{\psi_{P,n}^K}_{s+2} \le C \myfor \tau \le t_n-t_0\le T.
\end{equation}

\begin{lemma}[Stability in $\psi_I^K$]\label{lemma-stab-psiI}
We have, for $0\le t_n-t_0\le T$,
\[
\normbig{\widehat{\psi}_{I,n}^K - \psi_{I,n}^K}_s \le C e_n^K.
\]
\end{lemma}
\begin{proof}
This follows immediately from the definitions of $\psi_{I,n}^K$ in \eqref{eq-split-new-psi}, $\widehat{\psi}_{I,n}^K$ in \eqref{eq-split-hat-psi} and $e_n^K$ in \eqref{eq-errormax}.
\end{proof}

\begin{lemma}[Stability in $v$ and $w$]\label{lemma-stab-vw}
Under the regularity assumptions \eqref{eq-regularity} and \eqref{eq-reg-numerical}, we have, for $\tau\le t_n-t_0\le T$,
\[
\normbig{\widehat{v}_{n}^K - v_{n}^K}_s \le C e_n^K \myand \normbig{\widehat{w}_{n}^K - w_{n}^K}_s \le C e_n^K.
\]
\end{lemma}
\begin{proof}
Recall that $v_n^K$ and $\widehat{v}_n^K$ are defined in \eqref{eq-split-new-vw} and \eqref{eq-split-hat-vw}. 
For the estimate of $v$, we use that (omitting all superscripts $K$)
\[
\normbig{\Ic \klabig{\abs{\widehat{\psi}_{I,n-1}}^2 - \abs{\psi_{I,n-1}}^2} }_s 
 \le C \normbig{\widehat{\psi}_{I,n-1}-\psi_{I,n-1}}_s \normbig{\psi_{I,n-1}+\widehat{\psi}_{I,n-1}}_s
\]
by $\abs{a}^2-\abs{b}^2 = \ReT((a-b)(\overline{a}+\overline{b}))$, by Lemma \ref{lemma-inter} and by the algebra property \eqref{eq-algebra}. This implies the stated estimate of $\norm{\widehat{v}_{n}^K - v_{n}^K}_s$ using the stability estimate for $\psi_I^K$ of Lemma \ref{lemma-stab-psiI} and the bounds \eqref{eq-bound-hat-noncfl} and \eqref{eq-reg-numerical-vwpsiI} on $\widehat{\psi}_{I,n-1}^K$ and $\psi_{I,n-1}^K$.
For the estimate of $w$, we use in addition the bounds \eqref{eq-regularityK-high} and \eqref{eq-reg-numerical} on $\ph^K(t_n)$ and $\ph_n^K$, respectively.
\end{proof}

\begin{lemma}[Stability in $\psi_P^K$]\label{lemma-stab-psiP}
Under the regularity assumptions \eqref{eq-regularity} and \eqref{eq-reg-numerical} and under the CFL condition \eqref{eq-cfl}, we have, for $0\le t_n-t_0\le T$,
\[
\normbig{\widehat{\psi}_{P,n}^K - \psi_{P,n}^K}_{s+2} \le C e_n^K.
\]
\end{lemma}
\begin{proof}
From Lemma \ref{lemma-inter} and from the bound of Lemma \ref{lemma-poisson} on the inverse of the discrete Laplace operator in \eqref{eq-split-new-psi} and \eqref{eq-split-hat-psi}, we get (omitting all superscripts $K$)
\begin{align*}
\normbig{\widehat{\psi}_{P,n} - \psi_{P,n}}_{s+2} &\le C \norm{\ph(t_n)-\ph_n}_s + C \normbig{\widehat{\psi}_{I,n}-\psi_{I,n}}_s\\
 &\qquad\qquad + C \normbig{ \kla{\widehat{v}_n-v_n} \phi\klabig{\iu\tau(\widehat{v}_n-v_n)} \e^{\iu \tau v_n} \e^{\iu\tau\Om^2} \widehat{\psi}_{I,n} }_s\\
 &\qquad\qquad + C \normbig{ v_n \phi\klabig{\iu\tau v_n} \e^{\iu\tau\Om^2} \klabig{\widehat{\psi}_{I,n}-\psi_{I,n}} }_s.
\end{align*}
The stated stability estimate then follows from the properties \eqref{eq-exp}, the bounds \eqref{eq-bound-hat-noncfl} and \eqref{eq-reg-numerical-vwpsiI} on $\widehat{\psi}_{I,n}^K$, $\widehat{v}_n^K$ and $v_n^K$, the stability estimates for $\psi_I^K$ and $v$ of Lemmas \ref{lemma-stab-psiI} and~\ref{lemma-stab-vw} and the algebra property \eqref{eq-algebra}.
\end{proof}

The stability in the auxiliary variables $\psi_I^K$, $v$, $w$ and $\psi_P^K$ of Lemmas \ref{lemma-stab-psiI}--\ref{lemma-stab-psiP} can then be used to show the following stability properties in the main variables $\ph$, $u$ and $\dot{u}$ of \eqref{eq-split-new}. Note that in these estimates, there is no constant in front of the principal part on the right-hand side anymore.

\begin{proposition}[Stability in $\ph$]\label{prop-stab-ph}
Under the regularity assumptions \eqref{eq-regularity} and \eqref{eq-reg-numerical}, we have, for $\tau\le t_n-t_0\le T$,
\[
\normbig{ \widehat{\ph}_{n+1}^K - \ph_{n+1}^K }_{s} \le \normbig{\ph^K(t_n) - \ph_n^K}_{s} + C \tau e_n^K.
\]
\end{proposition}
\begin{proof}
Recall that $\ph_{n+1}^K$ and $\widehat{\ph}_{n+1}^K$ are defined in \eqref{eq-split-new-1} and \eqref{eq-split-hat-1}. 
Writing (omitting all superscripts $K$)
\[
\e^{-\iu\tau\widehat{v}_n} \ph(t_n)- \e^{-\iu\tau v_n} \ph_n = \klabig{1-\iu\tau\widehat{v}_n\phi(-\iu\tau\widehat{v}_n)} (\ph(t_n)-\ph_n) + \klabig{\e^{-\iu\tau\widehat{v}_n} - \e^{-\iu\tau v_n}} \ph_n,
\]
we get (omitting again all superscripts $K$)
\begin{align*}
\norm{ \widehat{\ph}_{n+1} - \ph_{n+1} }_{s} &\le \norm{ \ph(t_n)-\ph_n }_s + \tau \normbig{ \widehat{v}_n \phi\klabig{-\iu\tau \widehat{v}_n} \klabig{\ph(t_n)-\ph_n} }_s\\
 &\qquad + \tau \normbig{ (\widehat{v}_n-v_n) \phi\klabig{-\iu\tau (\widehat{v}_n-v_n)} \e^{-\iu\tau v_n} \ph_n }_s \\
 &\qquad + \tau^2 \normbig{(\widehat{v}_n-v_n) \phi\klabig{-\iu\tau(\widehat{v}_n-v_n)} \e^{-\iu\tau v_n}  \widehat{w}_n \phi\klabig{-\iu\tau^2 \widehat{w}_n} \widehat{\psi}_{I,n}}_s\\
 &\qquad + \tau \normbig{ \e^{-\iu\tau v_n} (\widehat{w}_{n}-w_n) \phi\klabig{-\iu\tau^2 (\widehat{w}_{n}-w_n)} \e^{-\iu \tau^2 w_n} \widehat{\psi}_{I,n}}_s\\
 &\qquad + \tau \normbig{ \e^{-\iu\tau v_n} w_n \phi\klabig{-\iu\tau^2 w_n} \klabig{\widehat{\psi}_{I,n}-\psi_{I,n}} }_s
\end{align*}
from Lemma \ref{lemma-inter}.
The stated stability estimate then follows as in the proof of the Lemma \ref{lemma-stab-psiP}, using in addition the bounds \eqref{eq-bound-hat-noncfl} and \eqref{eq-reg-numerical-vwpsiI} on $\widehat{w}_n^K$ and $w_n^K$ and the stability estimate for $w$ of Lemma \ref{lemma-stab-vw}.
\end{proof}

\begin{proposition}[Stability in $u$ and $\dot{u}$]\label{prop-stab-u}
Under the regularity assumptions \eqref{eq-regularity} and \eqref{eq-reg-numerical} and under the CFL condition \eqref{eq-cfl}, we have, for $0\le t_n-t_0\le T$,
\[
\normvbig{\klabig{\widehat{u}_{n+1}^K,\widehat{\dot{u}}_{n+1}^K} - (u_{n+1}^K,\dot{u}_{n+1}^K)}_{s} \le \normvbig{\klabig{u^K(t_{n}),\dot{u}^K(t_{n})} - (u_{n}^K,\dot{u}_{n}^K)}_{s} + C \tau e_n^K.
\]
\end{proposition}
\begin{proof}
We start with some properties of the matrix $R$ of \eqref{eq-R} that describes the time-discrete evolution in $u$ and $\dot{u}$, see \eqref{eq-split-new-1} and \eqref{eq-split-hat-1}. 
The first property is 
\[
\normvbig{ R(\tau) \klabig{\begin{smallmatrix} v\\ \dot{v} \end{smallmatrix}} }_s \le (1+\tau) \normvbig{ \klabig{\begin{smallmatrix} v\\ \dot{v} \end{smallmatrix}} }_s
\]
of Lemma \ref{lemma-bound-R}.
The second property that we need is
\begin{equation}\label{eq-bound-R-1}
\normvbig{ \klabig{ R(\tau)-1} \klabig{\begin{smallmatrix} v\\ \dot{v} \end{smallmatrix}} }_s \le C \tau \normvbig{ \klabig{\begin{smallmatrix} v\\ \dot{v} \end{smallmatrix}} }_{s+1}.
\end{equation}
This follows from $\abs{\cos(\xi)-1} \le \abs{\xi}$, $\abs{\tau\sinc(\xi)}\le \tau$ and $\abs{\tau^{-1} \xi \sin(\xi)} \le \tau^{-1} \abs{\xi}^2$ for $\xi\in\mathbb{R}$. With these properties of $R$, the claimed stability estimate follows from Lemma \ref{lemma-stab-psiP} on the stability in $\psi_P^K$ and from the bounds \eqref{eq-bound-hat-cfl} and \eqref{eq-reg-numerical-psiP} on $\widehat{\psi}_{P,n}^K$ and $\psi_{P,n}^K$, respectively.
\end{proof}

\subsection{Local error bounds}\label{subsec-localerror}

We study the local errors \eqref{eq-errors-local} of the scheme \eqref{eq-split-new}. Recall that the intermediate solution appearing in these differences is given by \eqref{eq-split-hat}.
For the exact solution of the spatial semi-discretization in the new variables \eqref{eq-zakK-new} we will use that, by the variation-of-constants formula,
\begin{equation}\label{eq-voc}\begin{split}
\ph^K(t_{n+1}) &= \e^{-\iu\tau\Om^2}\ph^K(t_n) - \iu \int_0^\tau \e^{-\iu(\tau-t)\Om^2} \Ic\klabig{ u^K(t_n+t) \ph^K(t_n+t) } \,\drm t\\ 
 & \qquad\qquad\qquad\quad\; - \iu \int_0^\tau \e^{-\iu(\tau-t)\Om^2} \Ic\klabig{ \dot{u}^K(t_n+t) \psi_I^K(t_n+t) } \,\drm t\\
\begin{pmatrix} u^K(t_{n+1}) \\ \dot{u}^K(t_{n+1}) \end{pmatrix} &= R(\tau) \begin{pmatrix} u^K(t_{n})\\ \dot{u}^K(t_{n}) \end{pmatrix} 
- \int_0^\tau R(\tau-t)  \begin{pmatrix} 0 \\ \Om^2 \Ic\klabig{ \abs{\psi_P^K(t_n+t)}^2 } \end{pmatrix} \, \drm t,
\end{split}\end{equation}
where $\psi_I^K$ and $\psi_P^K$ are as in \eqref{eq-zakK-new-2}.

As in the previous section on the stability of the method, we first study the local error in the auxiliary variables $\psi_I^K$, $v^K$, $w^K$ and $\psi_P^K$.

\begin{lemma}[Local error in $\psi_I^K$]\label{lemma-local-psiI}
Under the regularity assumption \eqref{eq-regularity}, we have, for $0\le t_n-t_0\le T$,
\[
\normbig{\psi_I^K(t_n) - \widehat{\psi}_{I,n}^K}_s \le C \tau.
\]
\end{lemma}
\begin{proof}
For $\psi_I^K$ from \eqref{eq-zakK-new-2} and $\widehat{\psi}_{I,n}^K$ from \eqref{eq-split-hat-psi}, we have
\[
\psi_I^K(t_n) - \widehat{\psi}_{I,n}^K = \sum_{j=1}^{n} \klabigg{ \int_{t_{j-1}}^{t_j} \ph^K(t)\,\drm t - \tau \ph^K(t_j) }.
\]
These are quadrature errors. Under the regularity assumption \eqref{eq-regularity} on the exact solution, we have $\norm{\partial_t \ph^K(t)}_s \le C$ by \eqref{eq-algebra}, \eqref{eq-zakK-new-1} and \eqref{eq-regularityK-high}, and hence the individual quadrature errors can be estimated in the norm $\norm{\cdot}_s$ by $C \tau^2$. 
\end{proof}

\begin{lemma}[Local errors in $v^K$ and $w^K$]\label{lemma-local-vw}
Under the regularity assumption \eqref{eq-regularity}, we have, for $\tau\le t_n-t_0\le T$,
\[
\normbig{u^K(t_n) - \widehat{v}_n^K}_s \le C \tau \myand
\normbig{\dot{u}^K(t_n) - \widehat{w}_n^K}_s \le C \tau.
\]
\end{lemma}
\begin{proof}
By the definitions \eqref{eq-split-hat-vw} of $\widehat{v}_n^K$ and $\widehat{w}_n^K$, we have
\begin{align*}
u^K(t_n) - \widehat{v}_n^K &= \klabig{1-\sinc(\tau\Om)} \klabig{ u^K(t_n) +\Ic \klabig{\abs{\widehat{\psi}_{I,n-1}^K}^2}}  + \sfrac12 \tau \sinc\klabig{\sfrac12 \tau\Om}^2 \dot{u}^K(t_n),\\
\dot{u}^K(t_n) - \widehat{w}_n^K &= \klabig{1-\sinc\klabig{\sfrac12\tau\Om}^2} \dot{u}^K(t_n)\\
 &\qquad + \klabig{1-\sinc(\tau\Om)} \Ic \klaBig{ \ReT\klabig{(\widehat{\psi}_{I,n}^K+\widehat{\psi}_{I,n-1}^K) \overline{\ph^K}(t_n)} }.
\end{align*}
Using that $\abs{1-\sinc(\xi)}\le \abs{\xi}$ and $\abs{1-\sinc\kla{\sfrac12\xi}^2}\le \abs{\xi}$, we get $\norm{(1-\sinc(\tau\Om))v}_s \le \tau \norm{v}_{s+1}$ and $\norm{(1-\sinc\kla{\sfrac12\tau\Om}^2)v}_s \le \tau \norm{v}_{s+1}$, respectively. The claimed estimates of the differences thus follow from the regularity properties \eqref{eq-regularityK-high} and \eqref{eq-bound-hat-noncfl} and the algebra property \eqref{eq-algebra} in combination with Lemma \ref{lemma-inter}. 
\end{proof}

\begin{lemma}[Local error in $\psi_P^K$]\label{lemma-local-psiP}
Under the regularity assumption \eqref{eq-regularity} and under the CFL condition \eqref{eq-cfl}, we have, for $0\le t_n-t_0\le T$,
\[
\normbig{\psi_P^K(t_n) - \widehat{\psi}_{P,n}^K}_{s+2} \le C \tau.
\]
\end{lemma}
\begin{proof}
For $n=0$, we have $\psi_P^K(t_0)=\widehat{\psi}_{P,0}^K$ by definition. For $n\ge 1$, 
we write
\begin{align*}
\theta^K(t_n) &= \Ic \klabig{ \iu \ph^K(t_n) + \psi_I^K(t_n) - u^K(t_n) \psi_I^K(t_n) },\\
\widehat{\theta}_n^K &= \Ic \klabig{ \iu\ph^K(t_{n}) + \widehat{\psi}_{I,n}^K - \widehat{v}_n^K \phi\klabig{\iu\tau\widehat{v}_n^K} \e^{\iu\tau\Om^2} \widehat{\psi}_{I,n}^K }.
\end{align*}
With this notation, the error $\psi_P^K(t_n) - \widehat{\psi}_{P,n}^K$ with $\psi_P^K$ from \eqref{eq-zakK-new-2} and $\widehat{\psi}_{P,n}^K$ from \eqref{eq-split-hat-psi} can be decomposed as
\begin{multline*}
\psi_P^K(t_n) - \widehat{\psi}_{P,n}^K = \klabig{\Om^2+1}^{-1} \klabig{\theta^K(t_n) -  \widehat{\theta}_n^K} \\
 + \klaBig{\klabig{\Om^2+1}^{-1} - \klabig{\Om^2\phi\klabig{\iu\tau\Om^2}+1}^{-1} } \widehat{\theta}_n^K.
\end{multline*}

(a) We first consider the second term in this decomposition.
From Lemma \ref{lemma-inter}, the bounds \eqref{eq-bound-hat-noncfl} of Lemma \ref{lemma-regularity-hat}, the bounds \eqref{eq-regularityK-high}, the property \eqref{eq-exp} and the algebra property \eqref{eq-algebra}, we get $\norm{\widehat{\theta}_n^K}_{s+2} \le C$. By Lemma \ref{lemma-poisson}, this implies for the second term in the above decomposition that 
\[
\normBig{\klaBig{\klabig{\Om^2+1}^{-1} - \klabig{\Om^2\phi\klabig{\iu\tau\Om^2}+1}^{-1} } \widehat{\theta}_n^K}_{s+2} \le C \tau.
\]

(b) For the first term in the above decomposition, we use the property
\begin{equation}\label{eq-exp2}
\norm{\phi(v)-1}_s \le \norm{v}_s \e^{C\norm{v}_s}, \qquad v\in H^s,
\end{equation}
which follows from the exponential series and the algebra property \eqref{eq-algebra},
and we use the property
\[
\normbig{\e^{\iu\tau \Om^2} v - v}_s \le \tau \norm{v}_{s+2}, \qquad v\in H^{s+2},
\]
which follows from $\abs{\e^{\iu \xi}-1} =\abs{\int_0^{\xi} \e^{\iu \eta}\,\drm \eta}\le \abs{\xi}$.
Together with Lemmas \ref{lemma-inter}, \ref{lemma-local-psiI} and \ref{lemma-local-vw} and the bounds \eqref{eq-regularityK-high} and \eqref{eq-bound-hat-noncfl}, this yields
\[
\normbig{\klabig{\Om^2+1}^{-1} \klabig{ \theta^K(t_n)-\widehat{\theta}_n^K}}_{s+2} \le C \normbig{\theta^K(t_n)-\widehat{\theta}_n^K}_s \le C \tau.
\]
Putting the estimates of (a) and (b) together yields the claimed estimate.
\end{proof}

The local error bounds in the auxiliary variables $\psi_I^K$, $v^K$, $w^K$ and $\psi_P^K$ yield the following local error bounds in the main variables $\ph^K$, $u^K$ and $\dot{u}^K$. 

\begin{proposition}[Local error in $\ph^K$]\label{prop-local-ph}
Under the regularity assumption \eqref{eq-regularity}, we have, for $\tau\le t_n-t_0\le T$,
\[
\normbig{\ph^K(t_{n+1}) - \widehat{\ph}_{n+1}^K}_s \le C \tau^2.
\]
\end{proposition}
\begin{proof}
(a) We first consider $\widehat{\ph}_{n+1}^K$ given by \eqref{eq-split-hat-1} and extract its dominant terms. We have
\begin{equation}\label{eq-phhat-dom}
\widehat{\ph}_{n+1}^K = \e^{-\iu\tau\Om^2} \Ic \klaBig{ \ph^K(t_n) -\iu \tau \widehat{v}_n^K \ph^K(t_n) -\iu\tau \widehat{w}_n^K \widehat{\psi}_{I,n}^K } -\iu \widehat{r}_n^K
\end{equation}
with 
\begin{align*}
\widehat{r}_n^K &= \tau \e^{-\iu\tau\Om^2} \Ic \klaBig{ \widehat{v}_n^K \klabig{ \phi(-\iu\tau\widehat{v}_n^K)-1} \ph^K(t_n) + \widehat{w}_n^K \klabig{\phi(-\iu\tau^2\widehat{w}_n^K)-1} \widehat{\psi}_{I,n}^K}.
\end{align*}
The remainder $\widehat{r}_n^K$ can be estimated with Lemma \ref{lemma-inter}, the property \eqref{eq-exp2} of the function $\phi$, the  algebra property \eqref{eq-algebra} and the bounds \eqref{eq-regularityK-high} and \eqref{eq-bound-hat-noncfl}. This yields 
\[
\normbig{\widehat{r}_n^K}_s \le C \tau^2.
\]

(b) We next consider $\ph(t_{n+1})$ as given by the variation-of-constants formula \eqref{eq-voc}. We have
\begin{equation}\label{eq-phK-dom}
\ph^K(t_{n+1}) = \e^{-\iu\tau\Om^2} \Ic \klaBig{ \ph^K(t_n) -\iu \tau u^K(t_n) \ph^K(t_n) - \iu\tau \dot{u}^K(t_n) \psi_{I}^K(t_n)} -\iu r_n^K
\end{equation}
with the quadrature error
\[
r_n^K = \int_0^\tau g^K(t) \,\drm t - \tau g^K(0), 
\]
where
\[
g^K(t) = \e^{-\iu(\tau-t)\Om^2} \Ic \klaBig{ u^K(t_n+t)\ph^K(t_n+t) + \dot{u}^K(t_n+t) \psi_I^K(t_n+t)}.
\]
In order to bound this quadrature error, we have to bound the derivative $\partial_{t} g^K$. For that purpose, we note that $\norm{\partial_t \ph^K}_s \le C$ by \eqref{eq-zakK-new-1}, the algebra property \eqref{eq-algebra} and the regularity property \eqref{eq-regularityK-high}. In addition, we use that $\norm{\partial_t \dot{u}^K}_s \le C$ by \eqref{eq-zakK-new-1} and $\norm{\partial_t \psi_I^K}_s \le C$ by \eqref{eq-zakK-new-2}.
Using \eqref{eq-regularityK-high} and Lemma \ref{lemma-inter}, this yields $\norm{\partial_t g^K}_s \le C$, and hence 
\[
\normbig{r_n^K}_s \le C \tau^2.
\]

(c) Subtracting the dominant terms \eqref{eq-phhat-dom} and \eqref{eq-phK-dom}, the claimed estimate follows from Lemmas \ref{lemma-local-psiI} and \ref{lemma-local-vw} on the local errors in $\psi_I^K$, $v^K$ and $w^K$.
\end{proof}

\begin{proposition}[Local error in $u^K$ and $\dot{u}^K$]\label{prop-local-u}
Under the regularity assumption \eqref{eq-regularity} and the CFL condition \eqref{eq-cfl}, we have, for $0\le t_n-t_0\le T$,
\[
\normvbig{\klabig{u^K(t_{n+1}),\dot{u}^K(t_{n+1})} - \klabig{\widehat{u}_{n+1}^K,\widehat{\dot{u}}_{n+1}^K} }_s \le C \tau^2.
\]
\end{proposition}
\begin{proof}
(a) We first extract the dominant term of $(u(t_{n+1}),\dot{u}(t_{n+1}))$ as given by the variation-of-constants formula \eqref{eq-voc}. Noting that
\[
- \int_0^\tau R(\tau-t) \begin{pmatrix} 0 \\ \Om^2 v \end{pmatrix} \,\drm t
= \klabig{R(\tau)-1} \begin{pmatrix} v \\ 0 \end{pmatrix},
\]
we have
\begin{equation}\label{eq-proof-localu-aux}
\begin{pmatrix} u^K(t_{n+1})\\ \dot{u}^K(t_{n+1}) \end{pmatrix} = R(\tau) \begin{pmatrix} u^K(t_{n})\\ \dot{u}^K(t_{n}) \end{pmatrix} + \klabig{R(\tau)-1} \begin{pmatrix} \Ic\klabig{\abs{\psi_P^K(t_n)}^2} \\ 0 \end{pmatrix} + \begin{pmatrix} r_n^K\\ \dot{r}_n^K\end{pmatrix}
\end{equation}
with 
\begin{align*}
\begin{pmatrix} r_n^K\\ \dot{r}_n^K\end{pmatrix} &= \int_0^\tau R(\tau-t) \begin{pmatrix} 0 \\ \Om^2 \Ic \klabig{ \abs{\psi_P^K(t_n)}^2 - \abs{\psi_P^K(t_n+t)}^2 } \end{pmatrix} \,\drm t.
\end{align*}
We then use that
\[
\normbig{\psi_P^K(t_n+t)-\psi_P^K(t_n)}_{s+2}= \normbigg{\int_{t_n}^{t_n+t} \ph^K(t')\,\drm t'}_{s+2} \le C \abs{t}
\]
by $\psi_P^K=\psi_I^K$ and \eqref{eq-regularityK-high}. From \eqref{eq-algebra} and \eqref{eq-regularityK-high} and from Lemmas \ref{lemma-inter} and \ref{lemma-bound-R}, we then get the bound
\[
\normvbig{\klabig{r_n^K,\dot{r}_n^K}}_s \le C \tau^2
\]
for the remainder in \eqref{eq-proof-localu-aux}. 

(b) We subtract the equation \eqref{eq-split-hat-1} for $\widehat{u}_{n+1}^K$ and $\widehat{\dot{u}}_{n+1}^K$ from \eqref{eq-proof-localu-aux}, which yields the error representation
\[
\begin{pmatrix} u^K(t_{n+1}) - \widehat{u}_{n+1}^K\\ \dot{u}^K(t_{n+1}) - \widehat{\dot{u}}_{n+1}^K \end{pmatrix} = \klabig{R(\tau)-1} \begin{pmatrix} v \\ 0 \end{pmatrix} + \begin{pmatrix} r_n^K\\ \dot{r}_n^K\end{pmatrix}. 
\]
with the error $v=\Ic\kla{\abs{\psi_P^K(t_n)}^2-\abs{\widehat{\psi}_{P,n}^K}^2}$ in $\abs{\psi_P^K}^2$.
The remainder $\kla{r_n^K,\dot{r}_n^K}$ has been estimated already in (a). 
To estimate the part with $R(\tau)-1$, we use \eqref{eq-bound-R-1} and $\norm{v}_{s+2}\le C\tau$ by the error bound of Lemma \ref{lemma-local-psiP}. 
Together with the above estimate of the remainder in \eqref{eq-proof-localu-aux}, this implies the error bound of the lemma.
\end{proof}

As \eqref{eq-split-new} is a two-term recursion, we also have to investigate the error of the starting approximation \eqref{eq-split-new-start} for $\ph$.

\begin{proposition}[Error of the starting approximation]\label{prop-error-starting}
Under the regularity assumption \eqref{eq-regularity}, we have
\[
\normbig{\ph^K(t_1)-\ph_1^K}_{s} \le C\tau.
\]
\end{proposition}
\begin{proof}
We start from the decomposition (note that $\ph^K=\partial_t \psi^K$ and $\tau\ph_1^K=\psi_1^K-\psi_0^K$)
\begin{equation}\label{eq-proof-start-aux}
\tau \klabig{\ph^K(t_1)-\ph_1^K} = \klabigg{\tau \ph^K(t_1) - \int_0^\tau \ph^K(t_0+t)\,\drm t} + \klabig{\psi^K(t_1)-\psi_1^K}.
\end{equation}

(a) The first term in the above decomposition is a quadrature error, which is bounded in the norm $\norm{\cdot}_s$ by $C\tau^2$ since $\norm{\partial_t \ph^K}_s \le C$ by \eqref{eq-algebra}, \eqref{eq-zakK-new-1} and \eqref{eq-regularityK-high}.

(b) As a preparation for estimating the second term in the decomposition \eqref{eq-proof-start-aux}, we show that
\[
\normbig{v_1^K-u_0^K}_s \le C \tau. 
\]
The proof of this estimate is based on the decomposition
\[
v_1^K - u_0^K = \klabig{v_1^K-\widehat{v}_1^K} + \klabig{\widehat{v}_1^K-u^K(t_1)} + \klabig{u^K(t_1)-u^K(t_0)}.
\]
The second term $\widehat{v}_1^K-u^K(t_1)$ in this decomposition can be estimated with Lemma \ref{lemma-local-vw}, and the third term can be estimated using $u^K(t_1)-u^K(t_0)=\int_0^\tau \dot{u}^K(t) \,\drm t$ and \eqref{eq-regularityK-high}. For the first term, we have
\[
v_1^K-\widehat{v}_1^K = \sinc(\tau\Om) \klabig{\widehat{u}_1^K-u^K(t_1)} - \sfrac12 \tau\sinc\klabig{\sfrac12 \tau\Om}^2 \klaBig{\widehat{\dot{u}}_1^K-\dot{u}^K(t_1)}
\]
since $\psi_{I,0}^K=\widehat{\psi}_{I,0}^K$, $u_1^K=\widehat{u}_1^K$ and $\dot{u}_1^K=\widehat{\dot{u}}_1^K$. We thus get from Proposition \ref{prop-local-u} a bound for this first term, and finally the above estimate of $v_1^K-u_0^K$.

(c) With this preparation, we now consider the second term in the decomposition \eqref{eq-proof-start-aux}. We first extract the dominant parts of $\psi^K(t_1)$ and $\psi_1^K$.
With the variation-of-constant formula for $\psi^K$ of \eqref{eq-zakK}, we get
\begin{equation}\label{eq-proof-starting-aux1}
\psi^K(t_{1}) 
 = \e^{-\iu\tau\Om^2}\psi^K(t_0) -\iu\tau \e^{-\iu\tau\Om^2} \Ic\klabig{ u^K(t_0) \psi^K(t_0) } -\iu r_0^K
\end{equation}
with the quadrature error
\[
r_0^K = \int_0^\tau \e^{-\iu(\tau-t)\Om^2} \Ic\klabig{ u^K(t_0+t) \psi^K(t_0+t) } \,\drm t - \tau \e^{-\iu\tau\Om^2} \Ic\klabig{ u^K(t_0) \psi^K(t_0) }.
\]
Using $\partial_t\psi^K=\ph^K$ and the bounds \eqref{eq-regularityK-high}, this quadrature error is seen to be bounded in the norm $\norm{\cdot}_{s}$ by $C\tau^2$. 
On the other hand, we have for $\psi_1^K$ (see \eqref{eq-split-new-start})
\begin{equation}\label{eq-proof-starting-aux2}
\psi_1^K = \e^{-\iu\tau\Om^2} \Ic\klabig{ \psi_0^K - \iu\tau u_0^K \psi_0^K} - \iu \widehat{r}_0^K
\end{equation}
with the remainder
\[
\widehat{r}_0^K = \tau \e^{-\iu\tau\Om^2} \Ic\klaBig{ v_1^K \klabig{\phi(-\iu\tau v_1^K)-1} \psi_0^K + \klabig{v_1^K - u_0^K} \psi_0^K }. 
\]
By \eqref{eq-vn-1}, we have $\norm{v_1^K}_s \le C$, and hence the term $\phi(-\iu\tau v_1^K)-1$ in the above remainder can be estimated with \eqref{eq-exp2}. The term $v_1^K - u_0^K$ has been estimated in (b). In this way we get
\[
\normbig{\widehat{r}_0^K}_s \le C \tau
\]
by the bounds \eqref{eq-regularityK-high} (with $t=t_0$), the algebra property \eqref{eq-algebra} and Lemma \ref{lemma-inter}. As the extracted dominant parts of $\psi^K(t_1)$ in \eqref{eq-proof-starting-aux1} and $\psi_1^K$ in \eqref{eq-proof-starting-aux2} are identical by the choice of initial values \eqref{eq-split-start}, we get the estimate as stated in the proposition.
\end{proof}

\subsection{Proof of Theorem \ref{thm-main}}\label{subsec-proof-main}

We finally put the local error bounds of Section \ref{subsec-localerror} and the stability estimates of Section \ref{subsec-stability} together to deduce the error bound of Theorem \ref{thm-main}. 

We start from the decomposition \eqref{eq-decom} and use Propositions \ref{prop-stab-ph} and \ref{prop-stab-u} on the stability of the method and Propositions \ref{prop-local-ph} and \ref{prop-local-u} on the local error. As long as the numerical solution satisfies the bounds \eqref{eq-reg-numerical}, this yields for $n\ge 1$
\[
e_{n+1}^K \le (1+C\tau) e_n^K + C \tau^2 
\]
for the global error $e_{n+1}^K$ of \eqref{eq-errormax}. For $n=0$, we get $e_1^K \le C \tau$ from Propositions \ref{prop-stab-u}, \ref{prop-local-u} and \ref{prop-error-starting} (recall that $u_0^K=u^K(t_0)$, $\dot{u}_0^K=\dot{u}^K(t_0)$ and $e_0^K=0$). 
Solving this recursion yields the error bound
\[
e_{n+1}^K \le C \tau.
\]
This error bound justifies that the regularity assumption \eqref{eq-reg-numerical} on the fully discrete solution also holds for $n+1$ instead of $n$. In fact, for sufficiently small $\tau$ and $0\le t_{n+1}-t_0\le T$, the regularity \eqref{eq-reg-numerical} follows from the error bound and the bounds \eqref{eq-regularityK-high} on $u^K$, $\dot{u}^K$ and $\ph^K$.

To get the statement of Theorem \ref{thm-main}, we have to translate the error bound in the new variables back to the original variables \eqref{eq-zakK} and \eqref{eq-split}. This is done with Lemmas \ref{lemma-stab-psiP} and \ref{lemma-local-psiP}, which yield the error bound
\[
\normbig{\psi^K(t_n)-\psi_{n}^K}_{s+2} = \normbig{\psi_P^K(t_n)-\psi_{P,n}^K}_{s+2} \le C \klabig{ \tau + e_n^K} \le C \tau.
\]
The proof of Theorem \ref{thm-main} is thus complete.

\section{Examples}\label{sec-exp}

\subsection{Numerical illustration of the error bound}

We consider the Zakharov system \eqref{eq-zak} and its discretization by the Lie--Trotter splitting \eqref{eq-split} in one space dimension ($d=1$). 
In order to illustrate the temporal error bound of Theorem \ref{thm-main}, we choose initial values in such a way that 
\[
\psi(t_0) \in H^{s+\si+2}, \qquad u(t_0) \in H^{s+\si+1}, \qquad \dot{u}(t_0) \in H^{s+\si} 
\]
for $s=1$ and $\si=2$, but not for $\si\ge 2.01$. More precisely, we choose
\begin{equation}\label{eq-init-numexp}
\psi(t_0) = w_5, \quad u(t_0) = w_4, \quad \dot{u}(t_0) = w_3,
\end{equation}
where
\[
w_{s'}(x) = \sum_{j\in\mathbb{Z}} \frac{2}{\max(\abs{j},1)^{s'+0.51}} \,\, \e^{\iu j\cdot x}, \qquad s'\ge 0.
\]
With these initial values, we apply the method \eqref{eq-split} with various time step-sizes and spatial discretization parameters. The temporal errors
\[
\normbig{\psi^K(t_n)-\psi_n^K}_{s+2}, \qquad \normbig{u^K(t_n)-u_n^K}_{s+1}, \qquad \normbig{\dot{u}^K(t_n)-\dot{u}_n^K}_{s}
\]
at time $t_n=t_0+\frac12 $ are plotted in Figure \ref{fig-error} versus the time step-size. A reference solution is computed using the standard fourth-order Runge--Kutta method with small step-size $\tau=10^{-7}$. As expected from Theorem \ref{thm-main}, we observe first-order convergence under the CFL condition \eqref{eq-cfl}.

\begin{figure}
\centering\includegraphics[trim=27 2 27 0,clip]{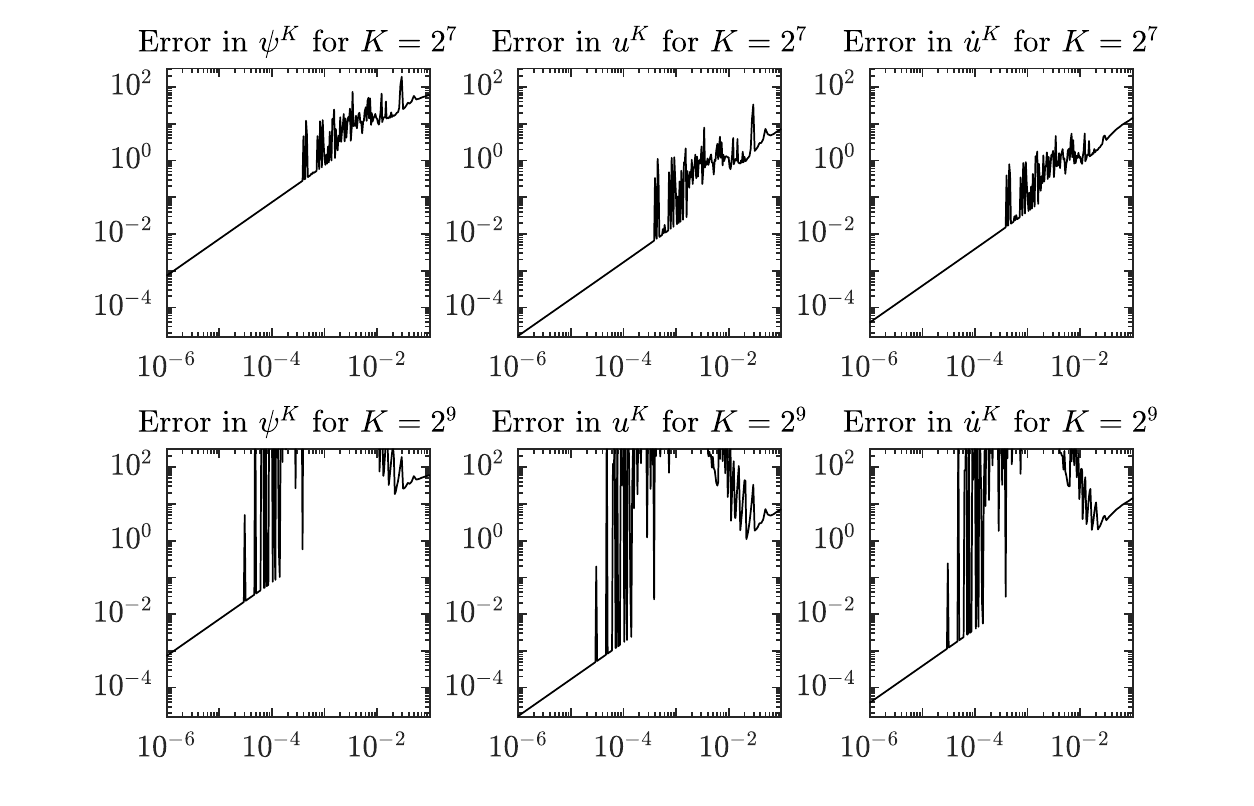}
\caption{Errors $\norm{\psi^K(t_n)-\psi_n^K}_{s+2}$ (left), $\norm{u^K(t_n)-u_n^K}_{s+1}$ (middle) and $\norm{\dot{u}^K(t_n)-\dot{u}_n^K}_{s}$ (right) vs.\ time step-size $\tau$ for $K=2^7$ (top) and $K=2^9$ (bottom).}\label{fig-error}
\end{figure}

\subsection{Direct estimates and loss of spatial regularity}

We consider the numerical method in the original variables \eqref{eq-split} and illustrate here, how direct estimates in these variables can lead to a loss of spatial regularity. In the following discussion, all mentioned bounds should hold and are supposed to hold \emph{uniformly} in the spatial discretization parameter $K$.

Assuming that $\psi_{n}^K$ is bounded in $H^{s+2}$, we get from equation \eqref{eq-split-1} for $u_{n+1}^K$ and $\dot{u}_{n+1}^K$ the desired stability estimate 
\[
\normvbig{\klabig{u_{n+1}^K,\dot{u}_{n+1}^K}}_{s'} \le \e^{C\tau} \normvbig{\klabig{u_{n}^K,\dot{u}_{n}^K}}_{s'} + C \tau \norm{\psi_n^K}_{s+2},
\]
but only for $s'\le s$ (for $s'>s$, we would loose the factor $\tau$ in front of $\psi_n^K$), see Lemma \ref{lemma-bound-R} and \eqref{eq-bound-R-1}. This means, that we can estimate $v_{n+2}^K$ given by \eqref{eq-vn-1} only in $H^{s+1}$. The equation \eqref{eq-split-1} for $\psi_{n+2}^K$ then suggests, however, that also $\psi_{n+2}^K$ can only be estimated in $H^{s+1}$, see \eqref{eq-exp}. We thus loose one order of the Sobolev space when compared to $\psi_n^K$ (which is bounded in $H^{s+2}$). This is the formal loss of spatial derivatives, which occurs in this type of naive estimate.

As we have seen in Section \ref{sec-proof}, this loss of derivatives can be circumvented by estimating in the new variables \eqref{eq-split-new} instead of the original variables \eqref{eq-split} and by imposing the CFL condition \eqref{eq-cfl}.

\subsection{Numerical loss of spatial regularity}

We illustrate numerically that the formal loss of regularity in the numerical solution \eqref{eq-split} described in the previous section is not just formal but is actually there, as soon as we don't impose the CFL condition \eqref{eq-cfl} (or use some additional filtering in the nonlinearity as it is done in \cite{Jin2004}). 

We use again the initial values \eqref{eq-init-numexp} in dimension $d=1$ such that $\psi_0^K\in H^{s+\si+2}$, $u_0^K\in H^{s+\si+1}$ and $\dot{u}_0^K\in H^{s+\si}$ with $s+\si=3$. We use a fixed spatial discretization parameter $K=2^8$ and four different time step-sizes
\[
\tau = 9.5\cdot 10^{-5}, \qquad \tau = 9.7\cdot 10^{-5}, \qquad \tau = 9.5\cdot 10^{-6}, \qquad \tau = 9.7\cdot 10^{-4}.
\]
The first of these time step-sizes is slightly below the limit $\tau = 2\pi/K^2 \approx 9.5874 \cdot 10^{-5}$ of the CFL condition \eqref{eq-cfl}, while the second one is slightly above. The third time step-size is far below this limit and the fourth one far above.

\begin{figure}
\centering\includegraphics[trim=24 9 34 6,clip]{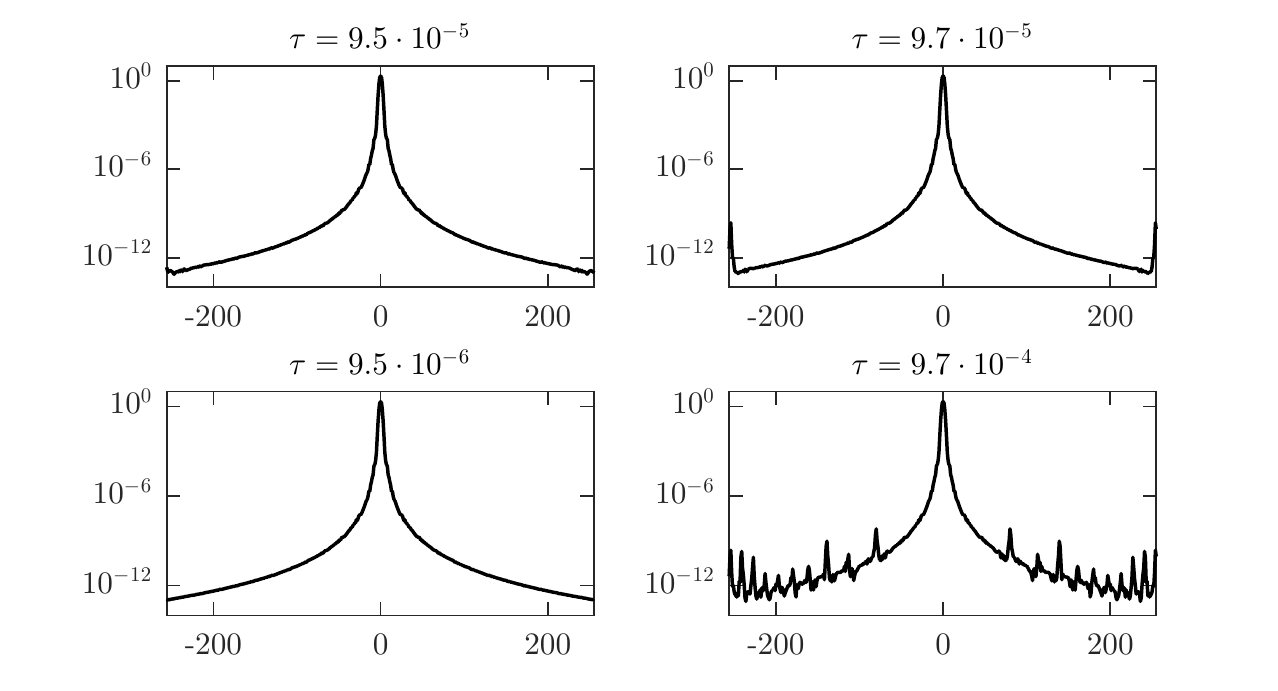}
\caption{Absolute values of the Fourier coefficients of $\psi_n^K$ at time $t_n\approx 0.15$ for $K=2^8$ and for four different time step-sizes $\tau$.}\label{fig-loss1}
\end{figure}

In Figure \ref{fig-loss1}, the absolute values of the Fourier coefficients $\psi_{n,j}^K$ of the numerical solution
\[
\psi_n^K(x)=\sum_{j=-K}^{K-1} \psi_{n,j}^K \, \e^{\iu j\cdot x}
\]
at time $t_n\approx t_0 + 0.15$ are plotted versus $j$. If the CFL condition \eqref{eq-cfl} does not hold, we observe an instability in high Fourier modes. More precisely, those Fourier modes are affected by an instability that violate the CFL condition, i.e., the modes $\psi_{n,j}^K$ with $\tau \abs{j}^2 > 2\pi$. The reason for this behaviour becomes clear from the error analysis in Section \ref{sec-proof}: for these high Fourier modes, we can't gain regularity of $\psi=\psi_P$ by means of Lemma \ref{lemma-poisson}. 

\begin{figure}
\centering\includegraphics[trim=34 8 33 6,clip]{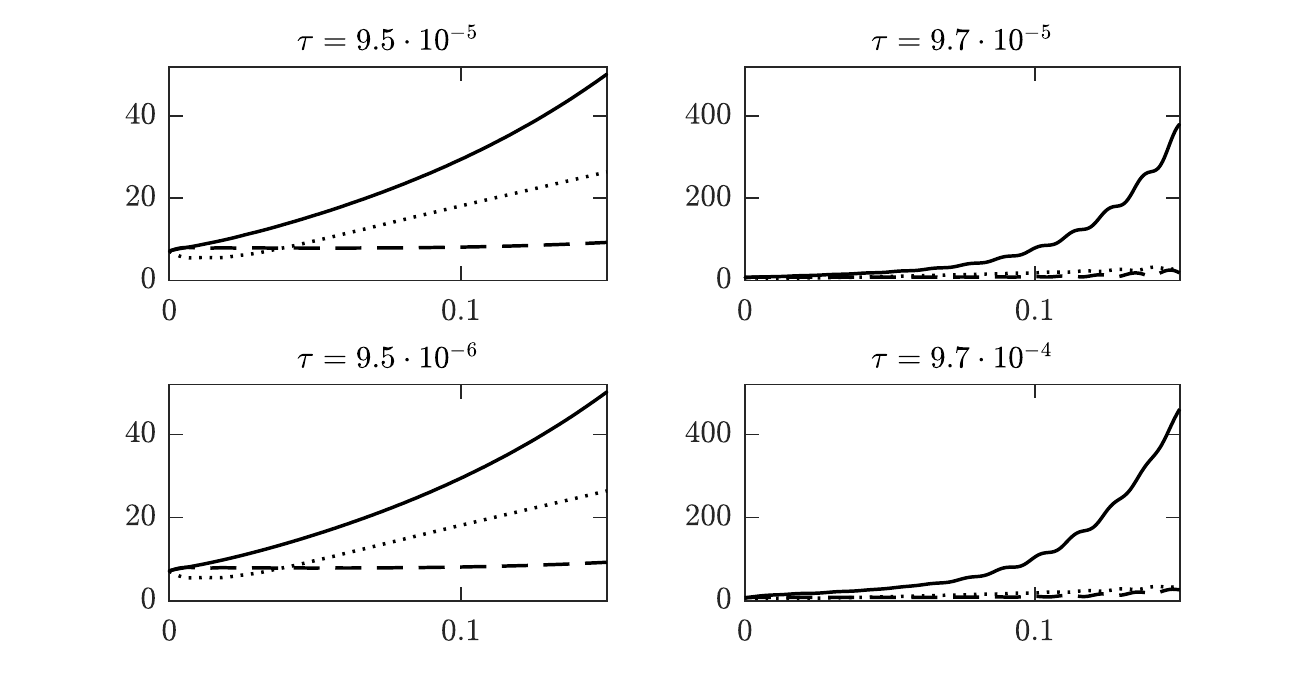}
\caption{$\norm{\psi_n^{K}}_{s+\si+2}$ (solid), $\norm{u_n^{K}}_{s+\si+1}$ (dashed) and $\norm{\dot{u}_n^{K}}_{s+\si}$ (dotted) vs.\ $t_n$ for $K=2^8$ and for four different time step-sizes $\tau$.}\label{fig-loss2}
\end{figure}

In Figure \ref{fig-loss2}, the evolution of the   norms $\norm{\psi_n^{K}}_{s+\si+2}$, $\norm{u_n^{K}}_{s+\si+1}$ and $\norm{\dot{u}_n^{K}}_{s+\si}$
along the numerical solution is plotted. We observe a significant growth if the CFL condition \eqref{eq-cfl} is violated. This illustrates the loss of spatial regularity in this case and the need for the CFL condition to avoid it.

\subsection*{Acknowledgement}

This work was supported by DFG project GA 2073/2-1 and by DFG collaborative research center 1114 ``Scaling cascades in complex systems''.

\end{document}